 \documentclass[hidelinks,onefignum,onetabnum]{siamart220329}
\headers{Stochastic Programming with Equality Constraints}{T. Lew, R. Bonalli, and M. Pavone}
\title{Sample Average Approximation for Stochastic Programming with Equality Constraints}
\author{Thomas Lew\thanks{Department of Aeronautics and Astronautics, Stanford University 
  (\email{thomas.lew@stanford.edu}, \email{pavone@stanford.edu}).}
\and Riccardo Bonalli\thanks{Laboratory of Signals and Systems, University of Paris-Saclay, CNRS, CentraleSup\'elec 
  (\email{riccardo.bonalli@l2s.centralesupelec.fr})}.
\and Marco Pavone\footnotemark[1]}

\usepackage{lipsum}
\usepackage{amsfonts}
\usepackage{graphicx}
\usepackage{epstopdf}
\usepackage{algorithmic}
\usepackage{amsopn}
\ifpdf
  \DeclareGraphicsExtensions{.eps,.pdf,.png,.jpg}
\else
  \DeclareGraphicsExtensions{.eps}
\fi
\newsiamremark{remark}{Remark}
\newsiamremark{hypothesis}{Hypothesis}
\crefname{hypothesis}{Hypothesis}{Hypotheses}
\newsiamthm{claim}{Claim}

\usepackage{url} %
\usepackage{wrapfig}

\usepackage{hyperref}
\usepackage{graphicx}
\usepackage{amsmath}
\usepackage{bm}
\usepackage{bbm}
\usepackage{mathrsfs}

\def\namedlabel#1#2{\begingroup
    #2%
    \def\@currentlabel{#2}%
    \phantomsection\label{#1}\endgroup
}
\newcommand\A[1]{\hyperref[A#1]{(A#1)}\xspace}
\newcommand\Aonetilde{\hyperref[Aonetilde]{$\widehat{\text{(A1)}}$}\xspace}
\newcommand\Afour{\hyperref[A4a]{(A4)}\xspace}
\newcommand\Afoura{\hyperref[A4a]{(A4a)}\xspace}
\newcommand\Afourb{\hyperref[A4b]{(A4b)}\xspace}

\newcommand\ALipschitz{\hyperref[A5]{(A5)}\xspace}
\newcommand\Afourtilde{\hyperref[A4tilde]{$\widehat{\text{(A4)}}$}\xspace}

\newcommand\RemarkRelaxingNec{\hyperref[remark:relaxing_necessary]{2.1}\xspace}

\newcommand\RemarkAtilde{\hyperref[remark:Atilde]{5.1}\xspace}

\def\Inf{\operatornamewithlimits{inf\vphantom{p}}}

\usepackage[font={small,it}]{caption}
\usepackage{color}

\usepackage{cleveref} %

\usepackage{xspace}

\usepackage{enumitem} %

\usepackage{cite} %

\usepackage{amssymb}%
\newcommand\rev[1]{\textcolor{black}{#1}}
\newcommand\revSecond[1]{#1} 
\newcommand\revThird[1]{#1} 
\newcommand\revFinal[1]{#1}

\newcommand\mydots{\hbox to 1em{.\hss.\hss.}}

\newcommand{\dd}{\textrm{d}} 
\newcommand{\dt}{\textrm{d}t}

\newcommand{\cost}{\ell}

\newcommand{\N}{\mathbb{N}}

\newcommand{\Prob}{\mathbb{P}}
\newcommand{\bProb}{\bar{\mathbb{P}}}
\newcommand{\bomega}{\bar\omega}
\newcommand{\bOmega}{\bar\Omega}

\newcommand{\K}{\mathcal{K}}

\newcommand{\G}{\mathcal{G}}

\newcommand{\B}{\mathcal{B}}

\newcommand{\U}{\mathcal{U}}

\newcommand{\E}{\mathbb{E}}
\newcommand{\R}{\mathbb{R}}

\newcommand{\prob}{\textbf{P}\xspace}
\newcommand{\sprob}{\textbf{SP}\xspace}
\newcommand{\ocp}{\textbf{OCP}\xspace}

\newcommand{\socp}{\textbf{SOCP}\xspace}

\newcommand{\cH}{\mathcal{H}}
\newcommand{\F}{\mathcal{F}}

\newcommand{\dHaus}{d_{\textrm{H}}}
\newcommand{\bD}{\mathbb{D}}

\usepackage{multirow}

\ifpdf
\hypersetup{
	pdftitle={Sample Average Approximation for Stochastic Programming with Equality Constraints},
  pdfauthor={T. Lew, R. Bonalli, and M. Pavone}
}
\fi

\begin{document}

\maketitle

\vspace{3mm}
\begin{abstract}
We revisit the sample average approximation (SAA) approach for non-convex stochastic programming. We show that applying the SAA approach to problems with expected value equality constraints does not necessarily result in asymptotic optimality guarantees as the \revThird{sample size} increases. To address this issue, we relax the equality constraints. Then, we prove the asymptotic optimality of the modified SAA approach under mild smoothness and boundedness conditions on the equality constraint functions. Our analysis uses random set theory and concentration inequalities to characterize the approximation error from the sampling procedure. We apply our approach \revSecond{and analysis} to the problem of stochastic optimal control for nonlinear dynamical systems \revSecond{under} external disturbances modeled by a Wiener process. \revSecond{Numerical results on relevant stochastic programs show the reliability of the proposed approach. Results on} a rocket-powered descent problem show that our computed solutions allow for significant uncertainty reduction \revSecond{compared to a deterministic baseline}.
\end{abstract}

%
%

\vspace{3mm}

\section{Introduction}
Sampling-based algorithms for stochastic programming have been successfully applied to a plethora of applications such as portfolio optimization \cite{Pagnoncelli2009}, resource management \cite{Kleywegt2002}, 
and controller design \cite{Campi2009}, among many others \cite{Homem2014}. 
In particular, the sample average approximation (SAA) is a popular approach that 
 consists of replacing the expected value of cost and constraints functions of the original \revThird{stochastic program} with empirical averages from \revThird{a random sample}. %
The resulting deterministic problem is then numerically solved via standard optimization techniques. 
The SAA approach is applicable to general non-convex stochastic programs, including with chance constraints \cite{Pagnoncelli2009} and risk measures \cite{Rockafellar2000}.

Many works have investigated the asymptotic convergence properties of the SAA approach 
\cite{Wang2008,Pagnoncelli2009,Shapiro2014,Bonnans2019}. The most fundamental property is consistency: As the \revThird{sample size} increases, solutions \revThird{to} the approximated problems should solve the original stochastic program. 
This consistency result is key to justifying the use of the SAA.  
However, these previous works only consider expected value inequality constraints.  
Considering equality constraints is crucial in many applications. For instance, stochastic optimal control problems in robotics and aerospace often include an expected value equality constraint on the final state of the system \cite{blackmore2011chance,BonalliLewESAIM2021}. 

\revSecond{Existing consistency results for the SAA approach \cite{Wang2008,Pagnoncelli2009,Shapiro2014,Bonnans2019} assume that by slightly perturbing optimal solutions, one can find feasible solutions that strictly satisfy inequality constraints. 
Thus, these consistency analyses
do not straightforwardly apply to 
equality-constrained programs. 
Indeed, applying the standard SAA approach to stochastic programs with expected value equality constraint may result in infeasibility for arbitrarily-large sample sizes, see Remark \RemarkRelaxingNec.} 
Problems with equality constraints are tackled in \cite{KrklecJerinki2019} by combining the SAA with the \rev{quadratic} penalty method. \rev{
This raises the question of whether  the SAA approach can be combined with other optimization schemes (such as interior point methods, sequential quadratic programming, or the augmented Lagrangian method that typically out-performs the penalty method) and still yield asymptotic optimality guarantees as the sample size increases.}
%
%
%
%

\newpage
In this work, we show that the SAA approach can be applied to %
stochastic programs with both expected value equality and inequality constraints and yield asymptotic optimality guarantees under mild assumptions. 
The key insight consists of relaxing the empirical equality constraints. 
By considering equality constraints, our analysis generalizes previous results in the stochastic programming literature  \cite{Wang2008,Pagnoncelli2009,Shapiro2014,Bonnans2019} and applies to pathological stochastic programs with inequality constraints, see Remark \RemarkAtilde. 
We rely on mild smoothness and boundedness assumptions on the functions defining the problem that can be verified prior to running the SAA approach.

Our analysis relies on random set theory and concentration inequalities. Specifically, we prove the almost-sure convergence of compact random level sets under a high-probability uniform error bound of the level set functions (Lemma \ref{lemma:ULLN_convCN}). Also, we derive a concentration inequality (Proposition \ref{prop:concent:F_alphaHolder}) that applies to classes of functions that are only H\"older continuous \A{2}, in contrast to prior work that assumes the Lipschitz continuity of the function class \cite{Bartlett2003,Koltchinskii2006,ShalevShwartz2009,wainwright_2019}. We believe these two intermediate results are of independent interest and will find other applications.

We apply the approach to the problem of stochastic optimal control for %
nonlinear systems characterized by a stochastic differential equation. Solving the general nonconvex formulation remains challenging, see  \cite{Peng1990,berret2020,BonalliLewESAIM2021,BonalliBonnet2022}. 
Hamilton-Jacobi-Bellman approaches \cite{Kushner2001} optimize over closed-loop policies and thus suffer from the curse of dimensionality due to the discretization of the state space.  %
In contrast, we optimize over open-loop control trajectories which allows for efficient numerical resolution. 
This approach is common in the literature \cite{blackmore2011chance,berret2020,BonalliLewESAIM2021} and allows for feedback by recursively solving the open-loop problem over time \cite{Mesbah2016}. %
We verify the approach on a Mars rocket-powered descent control problem and obtain solutions that minimize both the fuel consumption and the variance of the final state error.  
The paper is organized as follows. In \textbf{Section \ref{sec:problem_formulation}}, we define the stochastic programming problem with expectation equality constraints,  propose a sampling-based approximation of the problem, and state our assumptions and main asymptotic optimality result (Theorem \ref{thm:main:h_lipschitz_bounded}). 
In \textbf{Section \ref{sec:proof_main}}, 
we prove Theorem  \ref{thm:main:h_lipschitz_bounded}. The analysis requires two intermediate steps. 
    First, in \textbf{Section \ref{sec:reformulation:rand_sets}}, we reformulate the problem and the sampling-based approximation using random compact sets. With this reformulation, we derive asymptotic optimality guarantees of the approach under convergence assumptions on the cost and feasible set approximations (Theorem \ref{thm:main:rand_sets}).  
Finally, we derive sufficient conditions for the convergence of the random level sets that define the feasible equality constraints sets (Lemma \ref{lemma:ULLN_convCN}).
    Second, in \textbf{Section \ref{sec:concentration}}, we derive a concentration inequality (Proposition \ref{prop:concent:unif_bounded}) 
that implies the satisfaction of a high-probability uniform error bound required in Lemma \ref{lemma:ULLN_convCN} under mild assumptions \hyperref[A2]{(A2-3)}. 
In \textbf{Section \ref{sec:inequality}}, we extend the analysis to problems with inequality constraints. 
In \textbf{Section \ref{sec:stochastic_control}},  
we apply the analysis to stochastic optimal control problems.
In \textbf{Section \ref{sec:numerical_results}}, we provide numerical results validating the proposed approach. 
We conclude in \textbf{Section \ref{sec:conclusions}} and discuss directions of future work.

\section{Problem formulation and main result}\label{sec:problem_formulation}
For clarity of exposition, we defer the analysis of problems with inequality constraints to Section \ref{sec:inequality}.   
Let $(\Omega,\G,\Prob)$ be a probability space, 
$n,d\in\N$, 
$\U\subset\R^d$ be a compact set, 
$f:\U\times\Omega\to\R$ and 
$h:\U\times\Omega\to\R^n$ be two Carath\'eodory functions\footnote{The cost function $f:\U\times\Omega\to\R$ is Carath\'eodory if $f(u,\cdot)$ is $\G$-measurable for all $u\in\U$ and $f(\cdot,\omega)$ is continuous $\Prob$-almost-surely, and similarly for the equality constraints function $h$.}. 
We consider the stochastic optimization problem
\begin{align*}
\prob: 
\ \inf_{u\in \U}
\ \E[f(u,\revThird{\omega})]
\ \ \ 
\text{s.t.}
\ \ 
\E[h(u,\revThird{\omega})]=0
\end{align*}
and assume that \prob is feasible. 

We consider the sample average approximation (SAA) that consists of replacing the expectations in \prob with approximations from \revThird{a sample of} independent and identically distributed (i.i.d.)\ \revThird{realizations} $\omega^i\in\Omega$. 
Let $(\bOmega,\bar\G,\bProb)$ denote the product probability space (where $\bOmega=(\Omega\times\Omega\times\dots)$, see Section \ref{sec:product_space_socp} for details) and $\bomega=(\omega^1,\dots)\in\bOmega$ denote a \revThird{sample}, such that each $\omega^i\in\Omega$ corresponds to an i.i.d.\ \revThird{realization} from $(\Omega,\G,\Prob)$.
Given a \revThird{sample size} $N\in\N$, we define the empirical estimates 
\begin{align}
\label{eq:fN_gN}
f_N:\U\times\bOmega&\to\R  
&&h_N:\U\times\bOmega\to\R^n 
\\
\nonumber
(u,\bomega)&\mapsto
\frac{1}{N}\sum_{i=1}^N f(u,\omega^i)
&&
\qquad\ 
(u,\bomega)\mapsto \frac{1}{N}\sum_{i=1}^N h(u,\omega^i).
\end{align}
Given a sequence of strictly positive scalars $(\delta_N)_{N\in\N}$, we consider the sampled problem
\begin{align*}
\sprob_N(\bomega): 
\ \inf_{u\in \U}
\ f_N(u,\bomega)
\ \ 
\text{s.t.}
\ \ 
\left\|h_N(u,\bomega)\right\|\leq\delta_N,
\end{align*}
\revSecond{where $\|\cdot\|$ denotes the Euclidean norm.}  
For any \revSecond{sample size} $N\in\N$ and \revSecond{sample} $\bomega\in\bOmega$, we denote by $(f_0^\star,S_0^\star)$ and by $(f_N^\star(\bomega),S_N^\star(\bomega))$ the optimal value and set of solutions to \prob and to $\sprob_N(\bomega)$, respectively. 

The first objective of this work is to provide reasonable sufficient assumptions on $f$ and $h$ that can be checked prior to applying the SAA approach and guarantee that $\bProb$-almost-surely, $\sprob_N(\bomega)$ is feasible for a \revThird{sample size $N$} large enough, and solutions to $\sprob_N(\bomega)$ converge to optimal solutions \revThird{to} $\prob$ as $N$ increases. 
Relaxing the constraints of the sampled problem \revThird{$\sprob_N(\bomega)$} 
is a critical algorithmic choice to yield \rev{these} guarantees.

\begin{wrapfigure}{R}{0.275\linewidth}
    \centering
\includegraphics[width=0.99\linewidth]{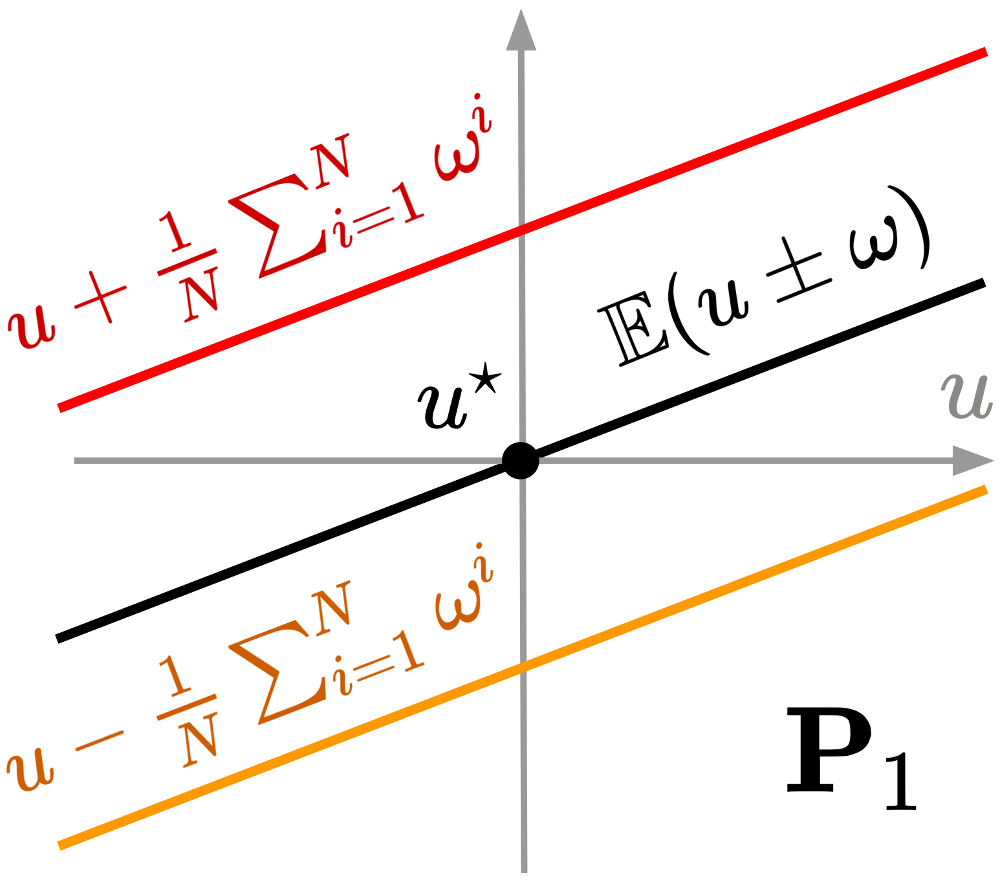}

\phantom{a}

\includegraphics[width=0.99\linewidth]{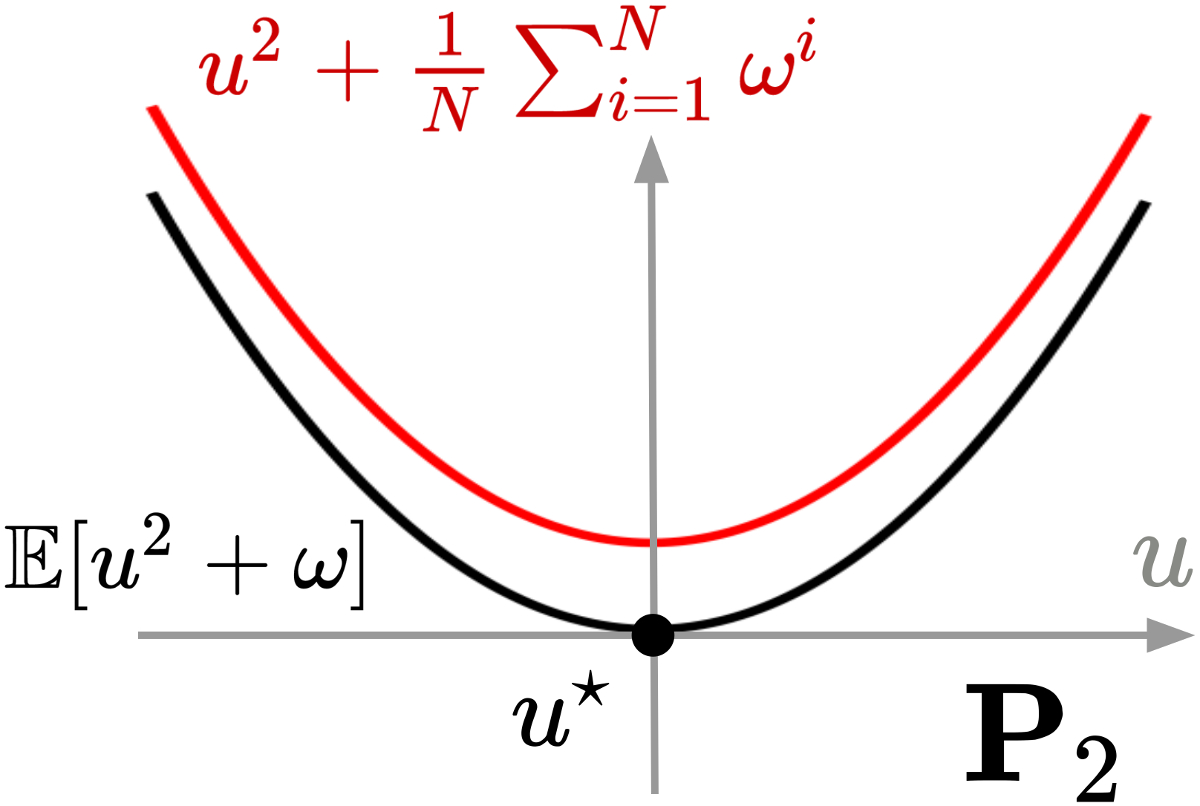}

\caption{\revThird{Two infeasible examples} for the standard SAA approach.
}%

    \vspace{-20mm}
\label{fig:SAAinfeasible}
\end{wrapfigure}

\textit{Remark \namedlabel{remark:relaxing_necessary}{2.1}} \revThird{(Infeasibility of the standard SAA)} 
Directly applying the standard SAA approach to \prob without relaxing the equality constraints by $\delta_N$ does not necessarily yield almost-sure asymptotic feasibility and optimality guarantees\revThird{, as we show next. Let $\U=\Omega=[-1,1]$ 
\revSecond{and} $\Prob$ be the Lebesgue measure normalized to $\Omega$, so that $\omega$ is uniformly distributed over  $\Omega$.} Consider the two problems
\begin{align*}
&\revThird{\prob_1:\inf_{u\in\revThird{\U}}\,|u|%
\ \ \ 
\text{s.t.}
\ \ 
\E[(u+\omega,u-\omega)]
=
(0,0),}
\\
&\revThird{\prob_2}:\inf_{u\in\revThird{\U}}\,|u|%
\ \ \ 
\text{s.t.}
\ \ 
\E[u^2+\omega]=0,
\end{align*}
\revThird{where $\prob_1$ is convex and $\prob_2$ has only one constraint. The solution to $\prob_1$ and $\prob_2$ is $u^\star=0$.}  However, \revThird{the two sampled problems} with strict equality constraint enforcement
\begin{align*}
&\revThird{\inf_{u\in\revThird{\U}}\,|u|%
\ \ \ 
\text{s.t.}
\ \ \left(\begin{matrix}
    u+\frac{1}{N}\sum_{i=1}^N\omega^i
    \\
    u-\frac{1}{N}\sum_{i=1}^N\omega^i
\end{matrix}\right)
=
\left(\begin{matrix}
    0\\0
\end{matrix}\right),}
\\
&\inf_{u\in\revThird{\U}}\,|u|%
\ \ \ 
\text{s.t. }
\ \ u^2+\frac{1}{N}\sum_{i=1}^N \omega^i=0
\end{align*}
\revThird{are infeasible with $\bProb$-probability $1$ and $\frac{1}{2}$, respectively, for all sample sizes $N\in\N$.}
The problem comes from the lack of constraint
qualification condition %
\cite{Shapiro2014}, see also \cite[Condition A]{Anisimov2000} and the hypotheses in \cite[Section 5.2.5.2]{Bonnans2019} which do not hold in this example, see Figure \ref{fig:SAAinfeasible}. 
Instead of making such assumptions that could potentially be restrictive, in this work, 
we assume that the equality constraints functions $h$ are smooth and bounded (see \hyperref[A2]{(A2-3)}) and relax the equality  constraints. %

By selecting the right dependency of $\delta_N$ on $N$, we show that the proposed stochastic programming approach yields asymptotic feasibility and optimality guarantees under %
assumptions of integrability \revThird{and smoothness} on $f$ and of smoothness and boundedness on $h$.

\begin{description} 
\item[\namedlabel{A1}{(A1)}]
There exists a $\G$-measurable function $d:\Omega\to\R$ such that $\E[d(\revThird{\omega})]<\infty$ and $|f(u,\omega)|\leq d(\omega)$ for all $u\in\U$ $\Prob$-almost 
surely.
\item[\namedlabel{Aonetilde}{$\widehat{\text{(\revThird{A1})}}$}] 
\revThird{$\Prob$-almost-surely, the map $u\mapsto f(u,\omega)$ is $\alpha$-H\"older continuous for some exponent $\alpha\in(0,1]$ and H\"older constant $L(\omega)$ such that }
$$
\revThird{|f(u_1,\omega)-f(u_2,\omega)|\leq L(\omega)\|u_1-u_2\|^\alpha
\quad \text{for all }\ u_1,u_2\in\U.}
$$
\item[\namedlabel{A2}{(A2)}] 
$\Prob$-almost-surely, the map $u\mapsto h(u,\omega)$ is $\alpha$-H\"older continuous for some exponent $\alpha\in(0,1]$ and H\"older constant $M(\omega)$ satisfying $\E[M(\revThird{\omega})^2]<\infty$, such that 
$\|h(u_1,\omega)-h(u_2,\omega)\|\leq M(\omega)\|u_1-u_2\|^\alpha$ for all $u_1,u_2\in\U$. 
\item[\namedlabel{A3}{(A3)}] $\Prob$-almost-surely, $\sup_{u\in\U} \|h(u,\omega)\|\leq \bar{h}$ for some constant $0\leq\bar{h}<\infty$. 
\end{description} 

\hyperref[A1]{(\revThird{A1-3})} are standard assumptions that are mild and can be checked prior to applying the proposed sampling-based approach. \revThird{Assuming $\alpha$-H\"older continuity for $\alpha<1$ in \Aonetilde and \A{2} is weaker than assuming that the maps $u\mapsto (f,g)(u,\omega)$ are Lipschitz continuous, as is typically done in the literature \cite{Bartlett2003,Koltchinskii2006,ShalevShwartz2009,wainwright_2019,Hu2020siam}. Making these weaker assumptions is necessary to apply our analysis to optimal control problems for dynamical systems described with a stochastic differential equation, see Section \ref{sec:stochastic_control} and \eqref{eq:xu_lip}.} \revThird{By multiplying $h$ with a smooth cutoff function, \A{3} is always satisfied.} %
We \revThird{further discuss assumptions \hyperref[A1]{(\revThird{A1-3})}} \revThird{(see Remark \ref{remark:unif_Markov}) }%
in the next section. 
Below, we state the main result of this paper. 

\begin{theorem}[Asymptotic optimality of the SAA approach for problems with expectation equality constraints]\label{thm:main:h_lipschitz_bounded}
Assume that the functions $f$ and $h$ satisfy \A{1},\A{2}, and \A{3}. 
Given any constants $\epsilon\in(0,\frac{1}{2})$ and $C>0$, define the sequence $(\delta_N)_{N\in\N}$ as 
\begin{equation}\label{eq:delta_N}
\delta_N=CN^{-(\frac{1}{2}-\epsilon)}.
\end{equation}
Then, $\bProb$-almost-surely, there exists a subsequence $\{N_k(\bar\omega)\}_{k\in\N}$ such that 
\begin{equation*}
f_{N_k(\bar\omega)}^\star(\bar\omega)\rightarrow f_0^\star
\quad
\text{and}
\quad
\mathbb{D}(S^\star_{N_k(\bar\omega)}(\bar\omega),S_0^\star)\rightarrow 0
\quad\ \text{as} \quad 
k\rightarrow\infty
\end{equation*}
with the distance function $\bD$ between the solution sets $S^\star_N(\bar\omega)$ and $S_0^\star$ defined in \eqref{eq:bD}.  
\revThird{If in addition, %
$f$ satisfies \Aonetilde, %
then $f_N^\star(\bomega)\to f_0^\star$ and $\bD(S^\star_N(\bomega),S_0^\star)\to 0$ as $N\to\infty$ $\bProb$-almost-surely.}
\end{theorem} 

This result implies that $\bProb$-almost-surely, $\sprob_N$ is feasible for a \revThird{sample size} $N$ large enough and solutions to $\sprob_N$ converge \revThird{(up to a subsequence)} to solutions to $\prob$ as $N$ increases. \revThird{If in addition, the cost function $f$ is $\alpha$-H\"older continuous under \Aonetilde, then the entire sequence of solutions to $\sprob_N$ converges to solutions to $\prob$.}

The remainder of this paper consists of proving Theorem \ref{thm:main:h_lipschitz_bounded} (Section \ref{sec:proof_main}), generalizing it to problems with inequality constraints (Section \ref{sec:inequality}), applying the proposed approach to stochastic optimal control problems (Section \ref{sec:stochastic_control}), and numerically evaluating the method  (Section \ref{sec:numerical_results}). 

 \newpage
\section{Proof of the main asymptotic optimality result (Theorem \ref{thm:main:h_lipschitz_bounded})}
\label{sec:proof_main}
\subsection{Reformulation with random sets and convergence analysis}\label{sec:reformulation:rand_sets}
We first observe that the feasible set of $\sprob_N$ is a random compact set and reformulate $\prob$ and $\sprob_N$ accordingly. Then, we derive  convergence guarantees under sufficient conditions. 

\subsubsection{Reformulating $\prob$ and $\sprob_N$ with random sets}
We denote by $\K$ the family of compact subsets of $\R^d$, 
by $\B(\K)$ the $\sigma$-algebra generated by the myopic topology on $\K$ \cite{Matheron1975,Molchanov_BookTheoryOfRandomSets2017}, and define $\K'=\K\setminus\{\emptyset\}$. For any $N\in\N$, we define
\begin{align}\label{eq:f0_g0}
f_0:\U\to\R: \, u\mapsto\E[f(u,\revThird{\omega})],
\qquad
h_0:\U\to\R^n: \, u\mapsto\E[h(u,\revThird{\omega})],
\end{align}
$f_N,h_N$ in \eqref{eq:fN_gN}, 
and the constraints sets 
\begin{subequations}
\label{eq:C0_CN}
\begin{gather}
\label{eq:C0}
C_0=\{u\in\U:h_0(u)=0\}
\\
\label{eq:CN}
C_N:\bOmega\to\K: \, 
\bomega\mapsto\{u\in\U:
\|h_N(u,\bomega)\|\leq\delta_N\}.
\end{gather}
\end{subequations}
Since $h$ is Carath\'eodory, 
$\U\in\K$, and \prob is feasible, $C_0\in\K'$. Also, $C_N$ is a random compact set \cite{Matheron1975,Aubin1990,Molchanov_BookTheoryOfRandomSets2017}, i.e., a measurable map from $(\bOmega,\G)$ to $(\K,\B(\K))$, see Appendix \ref{apdx:random_sets}. We will prove under mild conditions that $C_N(\bomega)\neq\emptyset$ for $N$ large enough $\bProb$-almost-surely (i.e., that $\sprob_N(\bomega)$ is feasible), see Lemma \ref{lemma:ULLN_convCN} and Proposition \ref{prop:concent:unif_bounded}. 

Then, \prob and $\sprob_N(\bomega)$ can be written as
\begin{align*}
\prob
:
\inf_{u\in C_0}
f_0(u)
\quad\text{and}\quad
\sprob_N(\bomega):
\inf_{u\in C_N(\bomega)}
f_N(u,\bomega).
\end{align*}
\subsubsection{General asymptotic convergence analysis}
We use the previous reformulation to prove a general convergence result under sufficient appropriate convergence of the approximated cost function $f_N$ and of the random constraints set $C_N$. 
To describe the convergence of the solutions sets, for any $A,B\in\K'$, we define 
\begin{equation}\label{eq:bD}
\bD(A,B)=\sup_{x\in A}
\Inf_{y\in B}
\|x-y\|
\end{equation}
and $\mathbb{D}(A,B)=\infty$ if $A$ or $B$ is empty.  
Although $\mathbb{D}$ is not a metric, it satisfies the triangle inequality and the property that $\mathbb{D}(A,B)=0$ implies $A\subseteq B$. We will use this last property to analyze the solution sets of \prob and $\sprob_N$. 

We describe the convergence of the constraints sets $C_N(\bomega)$ using the Hausdorff metric $\dHaus:\K\times\K\to[0,\infty)$ defined for any $A,B\in\K$ as 
\begin{equation}\label{eq:dH}
\dHaus(A,B)
=
\max(
\bD(A,B), \,
\bD(B,A)
).
\end{equation}
When considering optimization problems, it is generally only possible to prove convergence with respect to $\bD$; it is generally not possible to prove that $\dHaus(S^\star_N(\bomega),S_0^\star)\to 0$ as $N\to\infty$, see e.g. \cite{Vogel2006}.  
Assuming the uniform convergence of the cost function and constraints sets, we prove that up to a subsequence,  $\bD(S^\star_N(\bomega),S_0^\star)\to 0$ as $N\to\infty$ $\bProb$-almost-surely. 
We represent the distance $\bD(S_N^\star(\bomega),S_0^\star)$ in Figure \ref{fig:D_metric} and its limit.

\begin{figure}[t]%
  \centering
\includegraphics[width=0.3\linewidth]{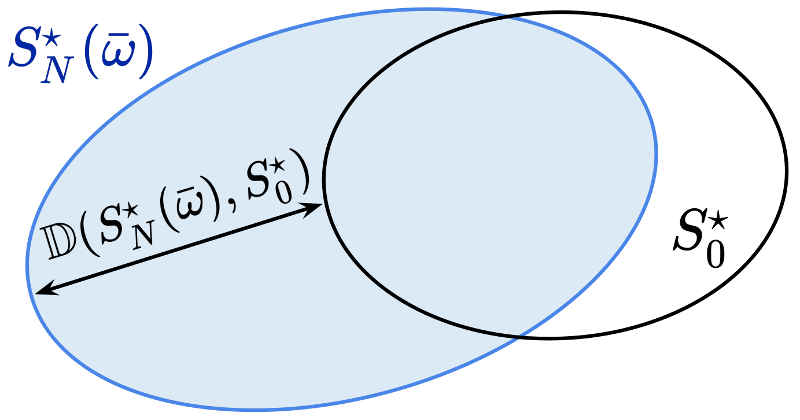}
\hspace{1cm}
\includegraphics[width=0.21\linewidth]{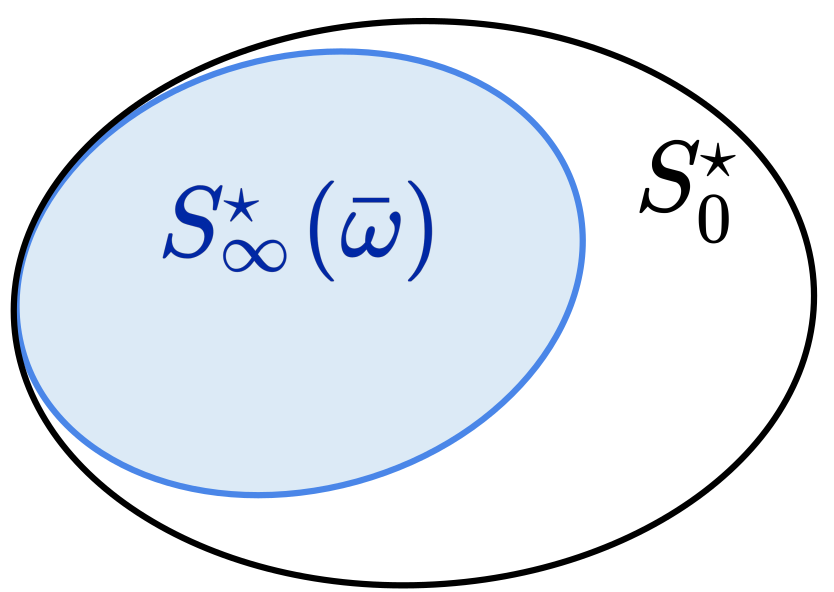}
  \caption{$\bD$-distance between the solution sets $S_0^\star$ and $S^\star_N(\bomega)$ to \prob and to $\sprob_N(\bomega)$, respectively. If $\bD(S^\star_N(\bomega),S_0^\star)\to 0$ as $N\to\infty$ as in Theorem \ref{thm:main:rand_sets}, then the limit of the solutions sets $S^\star_N(\bomega)$ is a subset of the solution set $S_0^\star$. Thus, in the limit as the \revThird{sample size} $N$ increases, solutions $u_N^\star(\bomega)\in S^\star_N(\bomega)$ to $\sprob_N(\bomega)$ are optimal solutions to \prob.}
\label{fig:D_metric}
\end{figure}

\begin{theorem}\label{thm:main:rand_sets}
Assume that the costs and constraints in $\prob$ and $\sprob_N$ satisfy 
\begin{enumerate}
\item \textit{(Uniform convergence of cost)} 
$\bProb$-almost-surely, 
\begin{equation}\label{eq:converges:fN:unif}
\sup_{u\in\U}|f_N(u,\bomega)-f_0(u)|\to 0\quad\text{as}\quad N\to\infty.
\end{equation}
\item \textit{(Myopic convergence of constraint set)} 
$\bProb$-almost-surely, 
\begin{equation}\label{eq:converges:CN}
\dHaus(C_N(\bomega),C_0)\to 0\quad\text{as}\quad N\to\infty.
\end{equation}
\end{enumerate}
Then, 
$\bProb$-almost-surely, there exists a subsequence $\{N_k(\bomega)\}_{k\in\N}$ such that 
\begin{subequations}
\label{eq:thm:conv_sol_sets}
\begin{align}
\label{eq:thm:conv_sol_sets:cost}
f_{N_k(\bomega)}^\star(\bomega)&\to f_0^\star
\ \ \,
\text{as} \quad 
k\rightarrow\infty,
\\
\label{eq:thm:conv_sol_sets:set}
\bD(S^\star_{N_k(\bomega)}(\bomega),S_0^\star)&\to 0
\quad\ 
\text{as} \quad 
k\rightarrow\infty.
\end{align}
\end{subequations}
\revThird{If in addition, %
$f$ satisfies \Aonetilde, %
then $f_N^\star(\bomega)\to f_0^\star$ and $\bD(S^\star_N(\bomega),S_0^\star)\to 0$ as $N\to\infty$ $\bProb$-almost-surely.}
\end{theorem}
The proof is provided in Appendix \ref{apdx:proof_thm:main:rand_sets}. 
It is inspired from \cite{Shapiro2014} but follows from different assumptions on the constraints sets. Specifically, in \eqref{eq:converges:CN}, we make a general assumption on the convergence of the constraints sets that allows considering equality constraints. By assuming that $\dHaus(C_N(\bomega),C_0)\to 0$ as $N\to\infty$ $\bProb$-almost-surely, Theorem \ref{thm:main:rand_sets} assumes that the sampled problems $\sprob_N$ are always feasible for a \revThird{sample size $N$} large enough.

The uniform convergence of $f_N$ to $f_0$ in \eqref{eq:converges:fN:unif} holds under \A{1}, see e.g. \cite[Theorem 5.40]{Bonnans2019} and \cite[Theorem 7.48]{Shapiro2014}. \A{1} \revThird{and \Aonetilde are} standard assumption\revThird{s} that \revThird{are} often satisfied in practice, see Section \ref{sec:stochastic_control}.  The difficulty consists of providing sufficient conditions for the $\bProb$-almost-sure convergence of the random feasible level sets in \eqref{eq:converges:CN}. We derive such conditions in the next section. %

\subsubsection{Convergence of random level sets} The feasible sets of $\sprob_N$ are random compact sets defined as the zero level sets of the empirical estimates $h_N$. Next, we derive sufficient conditions for the $\bProb$-almost-sure  convergence of these sets. %

\begin{lemma}[Almost-sure convergence of random level sets] \label{lemma:ULLN_convCN}
Let  $(\epsilon_N)_{N\in\N}$ and $(\revThird{\beta_N})_{N\in\N}$ be two sequences such that 
$\epsilon_N\to 0$ as $N\to\infty$,  
$\sum_{N=1}^\infty\revThird{\beta_N}<\infty$, and
\begin{equation}
\label{eq:unif_gN_epsN_alphaN}
\bProb\left(
\sup_{u\in\U}\|h_N(u,\bomega)-h_0(u)\|\leq \epsilon_N
\right)
\geq 1-\revThird{\beta_N}
\quad
\text{for all }N\in\N.
\end{equation}
Let $(\delta_N)_{N\in\N}$ be a sequence of scalars such that
$$
\delta_N>0\ \text{for all $N\in\N$},
\qquad 
\delta_N\to 0\ \text{as $N\to\infty$},
\qquad 
\frac{\epsilon_N}{\delta_N} \to 0\ \text{as $N\to\infty$}.
$$ 
Assume that there exists $u^\star\in\U$ such that $h_0(u^\star)=0$. 
Define the compact set $C_0\subseteq\U$ and the sequence of random compact sets $C_N:\ \bOmega\to\K$ as in \eqref{eq:C0_CN}. 

Then, $\bProb$-almost-surely, $C_N(\bomega)\neq\emptyset$ for all $N\geq N(\bomega)$ with $N(\bomega)$ large enough and 
\begin{equation}\label{eq:lemma:ULLN_convCN}
\dHaus(C_N(\bomega),C_0)\to 0\ \text{as}\  N\to\infty.
\end{equation}
\end{lemma}

The proof is provided in Appendix \ref{apdx:random_sets} and uses a result that provides sufficient conditions for the myopic convergence of random compact sets \cite[Proposition 1.7.23]{Molchanov_BookTheoryOfRandomSets2017}. 
Thanks to Lemma \ref{lemma:ULLN_convCN}, the condition \eqref{eq:converges:CN} in Theorem \ref{thm:main:rand_sets} can be satisfied using a high-probability uniform bound \eqref{eq:unif_gN_epsN_alphaN} on the error $h_N(u,\bomega)-h_0(u)$.

\subsection{Proving \eqref{eq:unif_gN_epsN_alphaN} via concentration inequalities}\label{sec:concentration} 
In this section, we use concentration inequalities to derive the high-probability uniform error bound \eqref{eq:unif_gN_epsN_alphaN} for the sampling-based approximation. %
We define the parametric function class
\begin{equation}\label{eq:fclass:H}
\cH\triangleq\{h(u,\cdot):\Omega\to\R^n, \ u\in\U\}.
\end{equation}
To derive bounds of the form of \eqref{eq:unif_gN_epsN_alphaN}, concentration inequalities exploit the structure of the considered function class. In this work, we make %
the smoothness and boundedness assumptions \hyperref[A2]{(A2-3)}.  
Assuming $\alpha$-H\"older continuity for $\alpha<1$ in \A{2} is weaker than assuming that the map $u\mapsto h(u,\omega)$ is Lipschitz continuous, as is typically done in the literature \cite{Bartlett2003,Koltchinskii2006,ShalevShwartz2009,wainwright_2019}. Making this weaker assumption is necessary to apply our approach and analysis to stochastic optimal control problems for nonlinear dynamical systems described with a stochastic differential equation, see Section \ref{sec:stochastic_control}.  %

\begin{remark}\label{remark:unif_Markov}
Markov's inequality and a standard symmetrization argument give
\begin{equation}\label{eq:Markov_and_RN}
\bar\Prob\left(
\sup_{u\in\U}\left\|\frac{1}{N}\sum_{i=1}^Nh(u,\omega^i)-\E[h(u,\revThird{\omega})]\right\|
\geq \epsilon
\right)
\leq
\frac{2}{\epsilon}
R_N(\cH)
\end{equation}
for any $\epsilon>0$ (see \cite[(4.17-4.18)]{wainwright_2019} and Appendix \ref{appdx:prop:concent:Fcentered_lip}), 
where $R_N(\cH)$ is defined in \eqref{eq:Rademacher} and denotes the Rademacher complexity of the function class $\cH$. 
Standard bounds for reasonable function classes give $R_N(\cH)\leq CN^{-\frac{1}{2}}$ for some constant $C$ \cite{ShalevShwartz2009,wainwright_2019}. However, combining such bounds with Markov's inequality does not provide a bound that is tight enough, i.e., a bound satisfying \eqref{eq:unif_gN_epsN_alphaN} with $\sum_{N=1}^\infty\revThird{\beta_N}<\infty$. 
Obtaining tighter bounds requires making stronger assumptions on the function class $\cH$. In this work, these stronger assumptions are %
\A{2} and \A{3}. %
\end{remark}

\subsubsection{Bounded differences}
McDiarmid's inequality \cite[Lemma 26.4]{ShalevShwartz2009}, also known as the bounded differences inequality \cite[Corollary 2.21]{wainwright_2019}, is a well-known result %
that allows deriving tight concentration inequalities for bounded function classes. Bounds resulting from this result are tighter than naively combining Markov's inequality with a symmetrization argument, see Remark \ref{remark:unif_Markov}. The following result is standard, see \cite[(4.16)]{wainwright_2019}.
\begin{corollary}\label{cor:concent:unif_bounded}
Assuming \A{3}, 
for any $N\in\N$, $\delta>0$, and $\revThird{\beta_N^\delta}=\exp\big(-\frac{N\delta^2}{2\bar{h}^2}\big)$,
$$
\sup_{u\in\U}\Big\|\frac{1}{N}\sum_{i=1}^Nh(u,\omega^i)-\E[h(u,\revThird{\omega})]\Big\|
\geq 
\bar\E\Big[\sup_{u\in\U}\Big\|\frac{1}{N}\sum_{i=1}^Nh(u,\omega^i)-\E[h(u,\revThird{\omega})]\Big\|\Big]
+
\delta
$$
\revThird{with $\bProb$-probability less than $\revThird{\beta_N^\delta}$.} 
\end{corollary}

\subsubsection{Dudley's entropy integral}%
\label{sec:concent:dudley_lip}
In this section, we bound the expectation term $\bar\E\sup_{u}\|\frac{1}{N}\sum_ih(u,\omega^i)-\E h(u,\revThird{\omega})\|$ in Corollary \ref{cor:concent:unif_bounded}. 
To do so, since the parameter set $\U$ is compact, we use Dudley's entropy integral (see \cite[Theorem 5.22]{wainwright_2019} and Theorem \ref{thm:dudley}).  %
This concentration inequality exploits the tight tails of sub-Gaussian processes and covering arguments via chaining. 
We combine this result with the smoothness assumption \A{2} of the equality constraints functions $h(u,\cdot)\in\cH$ to obtain a tight concentration inequality for our problem setting. 

\begin{proposition}[Concentration for $\alpha$-H\"older continuous function classes]\label{prop:concent:F_alphaHolder}
Assume that $h$ satisfies \A{2} for some exponent $\alpha\in(0,1]$ and H\"older constant $M(\omega)$. Let 
$D\triangleq 2\sup_{u\in\U}\|u\|$, 
$C=32$,  
$u_0\in\U$, 
and 
$\Sigma_0$ denote the covariance matrix of $h(u_0,\cdot)$. %
Then, %
$$
\bar\E\bigg[\sup_{u\in\U}\bigg\|\frac{1}{N}\sum_{i=1}^Nh(u,\omega^i)-\E[h(u,\revThird{\omega})]\bigg\|\bigg]
\leq 
\frac{1}{\sqrt{N}}
\bigg(
\frac{8CD^{\frac{\alpha+1}{2}}d^{\frac{1}{2}}n^{\frac{3}{2}}\E[M^2]^{\frac{1}{2}}}{\alpha^{\frac{1}{2}}}
+
\text{Trace}\left(\Sigma_0\right)^{\frac{1}{2}}
\bigg).
$$
\end{proposition}
The proof is provided in Appendix \ref{sec:proof:prop:concent:F_alphaHolder}. 
In contrast to existing work in the literature that assumes Lipschitz continuity \cite{Bartlett2003,Koltchinskii2006,ShalevShwartz2009,wainwright_2019}, Proposition \ref{prop:concent:F_alphaHolder} also gives a bound for function classes that are only $\alpha$-H\"older continuous for values of $\alpha<1$, see \A{2}.

\subsubsection{Concentration using bounded differences and Dudley's bound}
Using Corollary  \ref{cor:concent:unif_bounded} and Proposition \ref{prop:concent:F_alphaHolder},  
we obtain a bound of the form of \eqref{eq:unif_gN_epsN_alphaN}.
\begin{proposition}[Concentration for $\alpha$-H\"older bounded function classes]\label{prop:concent:unif_bounded}
Assume that $h$ satisfies \A{2} and \A{3}. 
With the notations of Proposition \ref{prop:concent:F_alphaHolder}, define the constant 
$\tilde{C}=\left(
8CD^{\frac{\alpha+1}{2}}d^{\frac{1}{2}}n^{\frac{3}{2}}\E\left[M^2\right]^{\frac{1}{2}}\alpha^{-\frac{1}{2}}
+
\text{Trace}\left(\Sigma_0\right)^{\frac{1}{2}}
\right)$. 
Given any $\epsilon>0$, define the two sequences $(\epsilon_N)_{N\in\N}$ and $(\revThird{\beta_N})_{N\in\N}$ by
\begin{equation}
\label{eq:eps_N_alpha_N}
\epsilon_N=2\tilde{C}N^{-\frac{1}{2}+\frac{\epsilon}{2}}
\ \text{ and }\ 
\revThird{\beta_N}=\exp\left(-\frac{\tilde{C}^2}{2\bar{h}^2}N^\epsilon\right).
\end{equation}
Then, 
$$
\bar\Prob\left(
\sup_{u\in\U}\left\|\frac{1}{N}\sum_{i=1}^Nh(u,\omega^i)-\E[h(u,\revThird{\omega})]\right\|
\leq\epsilon_N
\right)
\geq
1-\revThird{\beta_N}
\quad
\text{for all }N\in\N.
$$
\end{proposition}
The proof is provided in Appendix \ref{sec:proof:prop:concent:unif_bounded}. \revThird{We obtain the following corollary. %
\begin{corollary}\label{cor:prop:concent:unif_bounded}
Assume that $h$ satisfies \A{2} and \A{3}. 
Define  
$\tilde{C}$ as in Proposition \ref{prop:concent:unif_bounded}.  
Let $\epsilon>0$, $\beta\in(0,1)$, and $N\in\N$ be such than $N\geq\epsilon^{-2}(\tilde{C}+\bar{h}(2\log(1/\beta))^{\frac{1}{2}})^2$. 
Then, 
$$\bar\Prob\left(
\sup_{u\in\U}\left\|\frac{1}{N}\sum_{i=1}^Nh(u,\omega^i)-\E[h(u,\revThird{\omega})]\right\|
\leq\epsilon
\right)
\geq
1-\beta.
$$\end{corollary}
}

By selecting $\epsilon\in(0,\revThird{\frac{1}{2}})$ in Proposition \ref{prop:concent:unif_bounded}, the sequences $(\epsilon_N)_{N\in\N}$ and $(\revThird{\beta_N})_{N\in\N}$ satisfy $\epsilon_N\to 0$ as $N\to\infty$ and 
$\sum_{N=1}^\infty\revThird{\beta_N}<\infty$. We use this fact in the next section.

\subsection{Proof of Theorem \ref{thm:main:h_lipschitz_bounded}}\label{sec:proof:thm:main:h_lipschitz_bounded}
By combining the concentration inequalities derived in this section with the results from Section \ref{sec:reformulation:rand_sets}, we prove Theorem \ref{thm:main:h_lipschitz_bounded}. 

\begin{proof}[Proof of Theorem \ref{thm:main:h_lipschitz_bounded} (asymptotic optimality of the SAA approach)]
\ \, First, \ the \\cost function $f$ satisfies the uniform law of large numbers in \eqref{eq:converges:fN:unif} thanks to \A{1}, see e.g. \cite[Theorem 5.40]{Bonnans2019} and \cite[Theorem 7.48]{Shapiro2014}. 
Second, Proposition \ref{prop:concent:unif_bounded} gives
\begin{equation*}
\bar\Prob\left(
\sup_{u\in\U}\left\|h_N(u,\bar\omega)-h_0(u)\right\|
\leq
\epsilon_N
\right)
\geq 
1-\revThird{\beta_N}
\quad\text{for all }N\in\N
\end{equation*}
for $\epsilon_N=2\tilde{C}N^{-\frac{1}{2}+\frac{\epsilon}{2}}$ and $\revThird{\beta_N}=\exp\left(-\frac{\tilde{C}^2}{2\bar{h}^2}N^{\epsilon}\right)$ for some constant $\tilde{C}$. 
Since $\epsilon\in(0,\frac{1}{2})$, 
$\frac{\epsilon_N}{\delta_N}=\frac{2\tilde{C}}{C}N^{\frac{\epsilon}{2}-\epsilon}\to 0$ as $N\to\infty$ and $\sum_{N=1}^\infty\revThird{\beta_N}<\infty$. 
Thus, by Lemma \ref{lemma:ULLN_convCN}, 
\begin{equation*}
\dHaus(C_N(\bar\omega),C_0)\to 0\ \text{as}\  N\to\infty
\end{equation*}
$\bProb$-almost-surely. 
The conclusion follows from Theorem \ref{thm:main:rand_sets}.
\end{proof}

\section{Extension to problems with expected value inequality constraints}\label{sec:inequality}
We generalize our analysis to account for $q\in\N$ expected value inequality constraints. We consider a  Carath\'eodory function $g:\U\times\Omega\to\R^q$ and the stochastic program
\begin{align*}
\prob^g: 
\ \inf_{u\in \U}
\ \E[f(u,\revThird{\omega})]
\ \ \,
\text{s.t.}
\ \ 
\E[g(u,\revThird{\omega})]\leq 0,
\ \ 
\E[h(u,\revThird{\omega})]=0.
\end{align*}
We define the true expectation function  
 and the empirical estimate
\begin{subequations}\label{eq:g0_gN}
\begin{align}
g_0:&\ \U\to\R^q, \ u\mapsto\E[g(u,\revThird{\omega})]
\\
g_N:&\ \U\times\bOmega\to\R^q, \ 
(u,\bomega)\mapsto
\frac{1}{N}\sum_{i=1}^N g(u,\omega^i).
\end{align}
\end{subequations}
Following \cite{Pagnoncelli2009,Shapiro2014}, we consider the following first set of assumptions 
on the inequality constraints function $g$.
\begin{description} 
\item[\namedlabel{A4a}{(A4a)}] %
There exists a $\G$-measurable function $e:\Omega\to\R$ such that $\E[e(\revThird{\omega})]<\infty$ and $\|g(u,\omega)\|\leq e(\omega)$ for all $u\in\U$ $\Prob$-almost 
surely.
\item[\namedlabel{A4b}{(A4b)}] There exists a solution $u^\star\in S^\star_0$ %
such that for any $\epsilon>0$, there exists $u\in\U$ such that $\|u-u^\star\|\leq\epsilon$ and $g_0(u)<0$.  
\end{description}
\Afoura implies the $\bProb$-almost-sure uniform convergence $\sup_{u\in\U}|g_N(u,\bomega)-g_0(u)|\to 0$ as $N\to\infty$.  
\Afourb is a standard constraints qualification condition   \cite{Pagnoncelli2009,Shapiro2014} that guarantees the $\bProb$-almost-sure feasibility of the sampled problems defined below for a \revThird{sample size $N$} large enough. \Afourb also guarantees the existence of a sequence of solutions \revThird{to} the sampled problems that converges to $u^\star$ as the \revThird{sample size} $N$ increases. 
Given a sequence of scalars $(\delta_N)_{N\in\N}$, we consider the sampled problems
\begin{align*}
\sprob_N^g(\bomega): 
\ &\inf_{u\in \U}
\ f_N(u,\bomega)
\ \ 
\text{s.t.}
\ \ 
g_N(u,\bomega)\leq 0,
\ \ 
&\hspace{-15mm}\left\|h_N(u,\bomega)\right\|\leq\delta_N.
\end{align*}  

\begin{wrapfigure}{R}{0.25\linewidth}
	\begin{minipage}{0.99\linewidth}
    \centering
\includegraphics[width=0.99\linewidth]{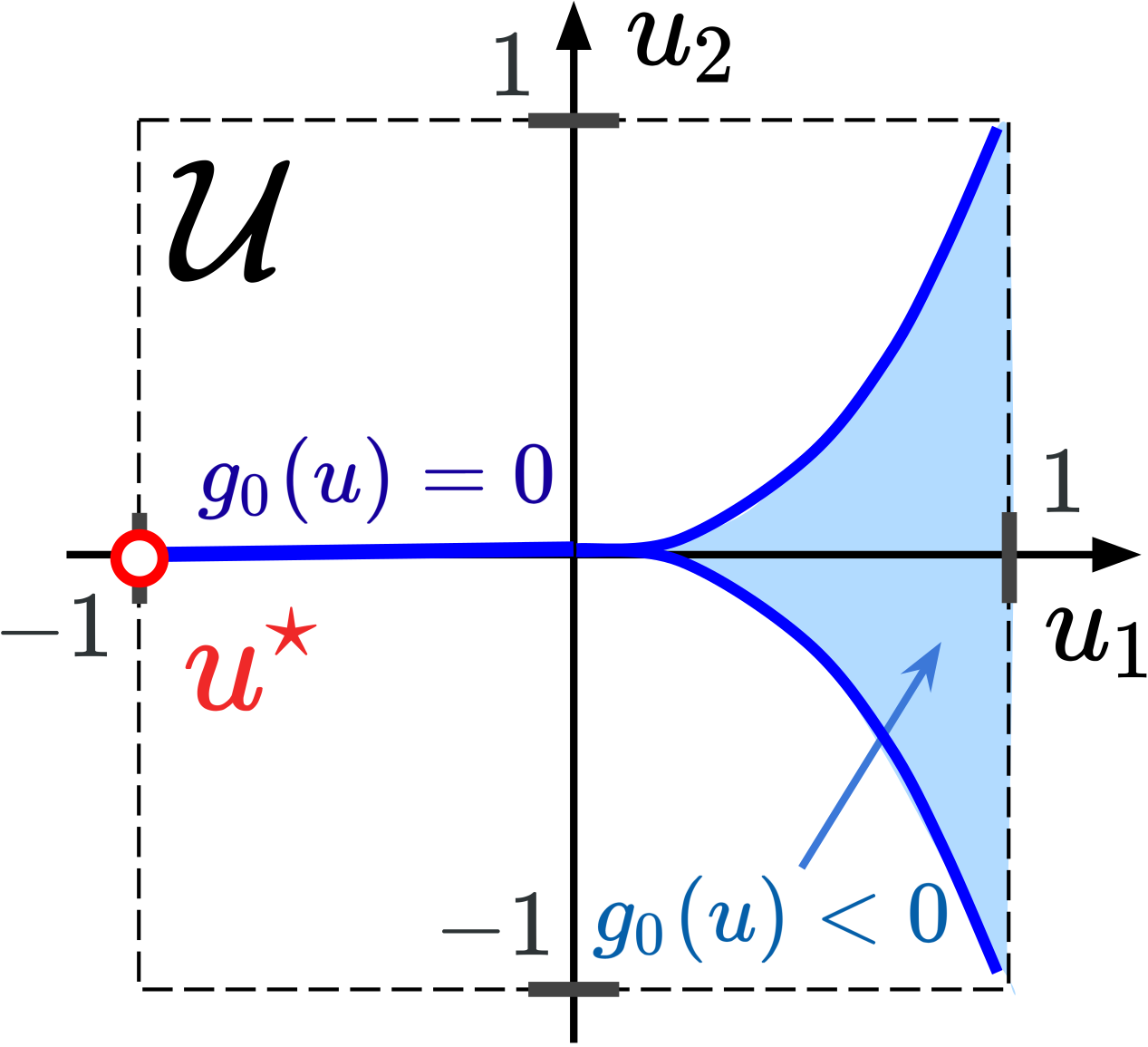}
\caption{Example that does not satify \Afourb: there exists no $u\in\U$ that satisfies $g_0(u)<0$ and $\|u-u^\star\|\leq 1$. \vspace{-7mm}}
\label{fig:g_qualification}
	\end{minipage}%
\end{wrapfigure}

\noindent 
Following \A{2} and \A{3}, 
we propose the alternative assumption \Afourtilde to tackle problems for which \Afourb does not hold. 
\begin{description} 
\item[\namedlabel{A4tilde}{$\widehat{\text{(A4)}}$}] $g$ satisfies \A{2} and \A{3}: there exists an exponent $\alpha\in(0,1]$, a H\"older constant $M(\omega)$ with $\E[M^2]<\infty$, and a bound $0\leq\bar{g}<\infty$ such that $\Prob$-almost-surely,
\begin{enumerate}[leftmargin=3mm]
\item $\|g(u_1,\omega)-g(u_2,\omega)\|\leq M(\omega)\|u_1-u_2\|^\alpha$ for all $u_1,u_2\in\U$, and 
\item $\sup_{u\in\U} \|g(u,\omega)\|\leq \bar{g}$. 
\end{enumerate}
\end{description}

\textit{Remark \namedlabel{remark:Atilde}{4.1.}} As represented in Figure \ref{fig:g_qualification}, the stochastic program with $\U=[-1,1]^2$, %
$f(u,\omega)=u_1$, and $g=(g_1,g_2)$ %
with $g_1(u,\omega)=u_2-u_1\max(u_1,0)$ and $g_2(u,\omega)=-u_2-u_1\max(u_1,0)$ has the unique optimal solution $u^\star=(-1,0)$. 
However, for $\epsilon\leq 1$, there exists no $u\in\U$ that is strictly feasible ($g_0(u)<0$) and $\epsilon$-close to the solution $u^\star$, so %
this problem does not satisfy \Afourb. However, this problem satisfies \Afourtilde. Compared to \Afourb, \Afourtilde represents smoothness and boundedness assumptions on the functions defining the problem that are simple to verify prior to applying the SAA approach.

Given $(\delta_N)_{N\in\N}$ and $\boldsymbol\delta_N=(\delta_N,\dots,\delta_N)\in\R^q$, we define the sampled problem
\begin{align*}
\widehat{\sprob_N^g}(\bomega): 
\ &\inf_{u\in \U}
\ f_N(u,\bomega)
\ \ 
\text{s.t.}
\ \ 
g_N(u,\bomega)\leq\boldsymbol\delta_N,
\ \ 
&\hspace{-15mm}\left\|h_N(u,\bomega)\right\|\leq\delta_N.
\end{align*}  
The next result generalizes Theorem \ref{thm:main:h_lipschitz_bounded} to problems with inequality constraints.

\begin{theorem}[Asymptotic optimality of the SAA approach]\label{thm:main:ineq}
Assume that the functions $f$ and $h$ satisfy \revThird{\A{1}, \A{2}, and \A{3}}. 
Given any constants $\epsilon\in(0,\frac{1}{2})$ and $C>0$, define the sequence $(\delta_N)_{N\in\N}$ as 
$$
\delta_N=CN^{-(\frac{1}{2}-\epsilon)}.
$$
For any $N\in\N$ and $\bomega\in\bOmega$, denote by $(f_0^\star,S_0^\star)$, 
$(f_N^\star(\bomega),S_N^\star(\bomega))$, and 
$(\widehat{f}_N^\star(\bomega),\widehat{S}_N^\star(\bomega))$ the optimal value and set of solutions to $\prob^g$, $\sprob_N^{\,g}(\bomega)$, and $\widehat{\sprob_N^{\,g}}(\bomega)$, respectively.

If $g$ satisfies \Afour, then, $\bProb$-a.s., there exists a subsequence $\{N_k(\bar\omega)\}_{k\in\N}$ such that 
\begin{equation*}
f_{N_k(\bar\omega)}^\star(\bar\omega)\rightarrow f_0^\star
\quad
\text{and}
\quad
\mathbb{D}(S^\star_{N_k(\bar\omega)}(\bar\omega),S_0^\star)\rightarrow 0
\quad\ \text{as} \quad 
k\rightarrow\infty
\end{equation*}
\revThird{and, if $f$ also satisfies \Aonetilde, then the entire sequence converges as $N\to\infty$.}

If $g$ satisfies \Afourtilde, then, $\bProb$-a.s., there exists a subsequence $\{N_k(\bar\omega)\}_{k\in\N}$ such that 
\begin{equation*}
\widehat{f}_{N_k(\bar\omega)}^\star(\bar\omega)\rightarrow f_0^\star
\quad
\text{and}
\quad
\mathbb{D}(\widehat{S}^\star_{N_k(\bar\omega)}(\bar\omega),S_0^\star)\rightarrow 0
\quad\ \text{as} \quad 
k\rightarrow\infty
\end{equation*}
\revThird{and, if $f$ also satisfies \Aonetilde, then the entire sequence converges as $N\to\infty$.}
\end{theorem} 
\begin{proof}
If $g$ satisfies \Afour, then the result follows from minor modifications to the proof of Theorem \ref{thm:main:rand_sets}. We refer to Appendix \ref{apdx:proof_thm:main:rand_sets} for details. 

Second, we assume that $g$ satisfies \Afourtilde and define the new feasible sets
\begin{subequations}
\begin{gather}
\tilde{C}_0=\{u\in\U:h_0(u)=0, 
\ 
g_0(u)\leq 0\}
\\
\tilde{C}_N:\bOmega\to\K: \, 
\bomega\mapsto\{u\in\U:
\|h_N(u,\bomega)\|\leq\delta_N, 
\ 
g_N(u,\bomega)\leq \boldsymbol\delta_N\}.
\end{gather}
\end{subequations}
By following the proof of Lemma \ref{lemma:ULLN_convCN}, we can show that $\bProb$-almost-surely, $\tilde{C}_N(\bomega)\neq\emptyset$ for all $N\geq N(\bomega)$ with $N(\bomega)$ large enough and 
$\dHaus(\tilde{C}_N(\bomega),\tilde{C}_0)\to 0$ as $N\to\infty$. %
The conclusion then follows from Theorem \ref{thm:main:rand_sets}. 

Indeed, to show that $\tilde{C}_N(\bomega)\neq\emptyset$ for $N\geq N(\bomega)$, we define the  infeasibility event
\begin{align*}
\tilde{B}_N
&=
\Big\{
\bomega\in\bOmega:\inf_{u\in\U}\max\big(
\|h_N(u,\bomega)\|,
g_{N,1}(u,\bomega),
\dots,
g_{N,q}(u,\bomega)
\big)
>\delta_N
\Big\}
\subseteq\tilde{B}_N^h\cup \tilde{B}_N^g
\end{align*}
where $\tilde{B}_N^h=\{\bomega\in\bOmega:
\|h_N(u^\star,\bomega)\|>\delta_N\}$ and $\tilde{B}_N^g=\{\bomega\in\bOmega:
g_N(u^\star,\bomega)>\boldsymbol\delta_N\}$ with $u^\star\in S_0^\star$. Then, by Boole's inequality,
$\bProb(\tilde{B}_N)\leq \bProb(\tilde{B}_N^h)+\bProb(\tilde{B}_N^g)$. 
By following the proof of Lemma \ref{lemma:ULLN_convCN}, thanks to \hyperref[A2]{(A2-3)},  \Afourtilde and Proposition \ref{prop:concent:unif_bounded}, we obtain that $\sum_{N=1}^\infty \bProb(\tilde{B}_N)<\infty$ so that $\tilde{C}_N(\bomega)=\emptyset$ only finitely many times $\bProb$-a.s. The proof that $\dHaus(\tilde{C}_N(\bomega),\tilde{C}_0)\to 0$ as $N\to\infty$ with $\bProb$-probability one also follows with minor adaptations of the proof of Lemma \ref{lemma:ULLN_convCN}. 
\end{proof}

\section{Application to stochastic optimal control}\label{sec:stochastic_control} 
We apply our methodology to general stochastic control problems with expected value inequality and equality constraints. 
Let $T>0$ be a finite time horizon and  $n,m\in\N$ be state and control dimensions. 
We consider as optimization parameter an open-loop control trajectory $\bm{u}=(u_0,\dots,u_{S-1})\in\R^{Sm}$ with $S\in\N$ stepwise-constant controls $u_s$ taking values in a compact set $U\subset\R^m$. Thus, the control space $\U\cong U\times\dots\times U\subset\R^{Sm}=\R^d$ is finite-dimensional and any control $u\in\U\subset L^2([0,T];U)$ takes the form
\begin{equation}\label{eq:control_finite}
u(t)=\sum_{s=0}^{S-1} u_s\mathbbm{1}_{\left[\frac{sT}{S},\frac{(s+1)T}{S}\right)}(t)
\quad \text{where } u_s\in U \ \text{for all } s=0,\dots,S-1
\end{equation}
and satisfies $\|u\|_{L^2([0,T];U)}^2=\frac{T}{S}\|\bm{u}\|^2$. Thus, in the remainder of this section, we implicitly identify $\U$ with $U^S$ and their corresponding norms. 
This finite-dimensional piecewise-constant control representation is common in the literature \cite{Kushner2001,Acikmese2007}, as it allows for simple numerical implementation. Note that by increasing $S$, any function in $L^2([0,T];U)$ can be approximated arbitrarily well by a function in $\U$. 

We model the uncertain dynamical system using a stochastic differential equation (SDE). Let $b\in C^0(\R^n\times U;\R^n)$ be a drift function 
and $\sigma\in C^0(\R^n\times U;\R^{n\times n})$ be a diffusion coefficient. We assume the following regularity condition: 
\begin{description} 
\item[\namedlabel{A5}{(A5)}] 
There exists a bounded constant $K\geq 0$ such that for all $x,y\in\R^n$ and $u,v\in U$,
$$
\|b(x,u)-b(y,v)\|
+
\|\sigma(x,u)-\sigma(y,v)\|
\leq 
K(\|x-y\|+\|u-v\|).
$$
\end{description}
By requiring additional appropriate assumptions (see  \cite{Yong1999} and \cite{Frankowska2019}), generalizing our analysis to time-varying uncertain drift and diffusion coefficients (i.e., $b,\sigma$ depend on time $t\in[0,T]$ and on uncertain parameters $\omega\in\Omega$) is straight-forward. 

Let $(\Omega,\G,\F,\Prob)$ be a filtered probability space with $W$ the standard $n$-dimensional Brownian motion and
$L^2_{\F}(\Omega;C([0,T];\R^n))$ the space of square-integrable $\F$-adapted stochastic processes with continuous sample paths taking values in $\R^n$. 
Given a control $u\in\U$ and an initial condition $x_0\in\R^n$, we define the SDE 
\begin{equation}
\label{eq:SDE}
\dd x_u(t)=b(x_u(t),u(t))\dt+\sigma(x_u(t),u(t))\dd W_t, 
\ t\in[0,T],
\quad x_u(0)=x_0.
\end{equation}
Under \ALipschitz, for every $u\in\U$, \eqref{eq:SDE} has a unique solution $x_u\in L^2_{\F}(\Omega;C([0,T];\R^n))$ up to stochastic indistinguishability. 

In Appendix \ref{sec:xu_continuous}, we show that up to a modification, 
for any $\alpha\in(0,\frac{1}{2})$ and $u,v\in\U$,  there exists a constant $C_\alpha(\omega)$ satisfying $\E[C_\alpha(\revThird{\omega})^2]<\infty$ such that, $\Prob$-almost-surely, 
\begin{equation}
\label{eq:xu_lip}
\sup_{0\leq t\leq T}\|x_u(t,\omega)-x_v(t,\omega)\|
\leq 
C_\alpha(\omega)
\|u-v\|_{L^2([0,T];U)}^\alpha.
\end{equation}
In particular, up to a modification, $\Prob$-almost-surely, the map $u\mapsto x_u(t,\omega)$ is continuous and satisfies \A{2}. 

Let $\ell\in C^0(\R^n\times U;\R)$ be a running cost function, $\varphi\in C^0(\R^n;\R)$ be a terminal cost function,  and 
$G\in C^0(\R^n;\R^q)$, 
$H\in C^0(\R^n;\R^n)$ be two bounded 
constraints functions.  
We assume the following condition:
\\
\textbf{\namedlabel{A6}{(A6)}}
There exists a bounded constant $L\geq 0$ such that for all $x,y\in\R^n$ and $u,v\in U$,
\begin{equation*}
    \begin{gathered}
\|H(x)-H(y)\|+\|G(x)-G(y)\|\leq L\|x-y\|,
\quad\  
\|H(x)\|+\|G(x)\|\leq L,
\\
\|\ell(x,u)-\ell(y,v)\|
+
\|\varphi(x)-\varphi(y)\|
\leq
L(\|x-y\|+\|u-v\|).
    \end{gathered}
\end{equation*}
We apply our approach to optimal control problems with state constraints. %
We define %
\begin{align*}
\ocp:\ 
&
\inf_{u\in\U}
\ \  
\E\bigg[ \int_{0}^{T} \ell(x_u(t,\revThird{\omega}),u(t))\, \dd t
+
\varphi(x_u(T,\revThird{\omega}))
\bigg]
\\
&\text{ s.t.}
\ \ \ \eqref{eq:SDE},
\quad 
\E\bigg[\sup_{t\in[0,T]} G(x_u(t,\revThird{\omega}))\bigg]\leq 0, 
\quad 
\E[H(x_u(T,\revThird{\omega}))]=0
\end{align*}
with the notation $\sup_{t} G(x_u(t,\revThird{\omega}))\triangleq (\sup_{t} G_1(x_u(t,\revThird{\omega})),\dots,\sup_{t} G_q(x_u(t,\revThird{\omega})))$. 
We assume that \ocp  
is feasible. 

Since the map $u\mapsto x_u(t,\omega)$ is continuous $\Prob$-almost-surely, the maps $f:\U\times\Omega\to\R$, 
$g:\U\times\Omega\to\R^q$, and $h:\U\times\Omega\to\R^n$ defined as 
\begin{subequations}\label{eq:fh_ocp}
\begin{align}
f(u,\omega)&=\int_0^T\ell(x_u(t,\omega),u(t))\dd t
+
\varphi(x_u(T,\omega))
\\
g(u,\omega)&=
\bigg(
	\sup_{t\in[0,T]} G_1(x_u(t,\omega)),\dots,\sup_{t\in[0,T]} G_q(x_u(t,\omega))
\bigg)
\\
h(u,\omega)&=H(x_u(T,\omega))
\end{align}
\end{subequations}
are Carath\'eodory (note that $\Prob$-almost-surely, $x_u$ has continuous sample paths, so that $g$ is well-defined). 
Thus, \ocp takes the form of \prob. This justifies applying the SAA approach to approximately solve \ocp.

We denote the product probability space by $(\bar\Omega,\bar\G,\bar\Prob)$ (see Section \ref{sec:product_space_socp} for details) and a \revThird{sample} by $\bar\omega=(\omega^1,\dots)\in\bar\Omega$. Given a number $N\in\N$ of \revThird{realizations} $\omega^i\in\Omega$ of the Brownian motion and a padding $\delta_N>0$,  we define the sampled problem
\begin{align*}
\socp_N(\bar\omega):\ 
&
\inf_{u\in\U}\ \ 
\frac{1}{N}\sum_{i=1}^N  
\left(
\int_{0}^{T} \ell(x_u(t,\omega^i),u(t))\, \dd t
+
\varphi(x_u(T,\omega^i))
\right)
\quad
\text{ s.t.}\ \ \ \eqref{eq:SDE},
\\
&
\frac{1}{N}\sum_{i=1}^N \sup_{t\in[0,T]}G(x_u(t,\omega^i))\leq\boldsymbol\delta_N,
\qquad
\left\|\frac{1}{N}\sum_{i=1}^N
H(x_u(T,\omega^i))\right\|\leq\delta_N
\end{align*}
where $\boldsymbol\delta_N=(\delta_N,\dots,\delta_N)\in\R^q$. 
Since the control space $\U\cong\R^{Sm}$ is compact and finite-dimensional and the constraint functions $G$ and $H$ are smooth and bounded, we apply Theorem \ref{thm:main:ineq} and obtain the following consistency result.

\begin{theorem}[Asymptotic optimality of the SAA approach for stochastic optimal control]\label{thm:main:socp}
\revThird{Assume that \ocp is feasible and that $(b,\sigma,\ell,H,G,\varphi,\ell)$ satisfy \ALipschitz and \A{6}.} 
Given any constant $\epsilon\in(0,\frac{1}{2})$ and  $C>0$, define the sequence $(\delta_N)_{N\in\N}$ as 
$\delta_N=CN^{-(\frac{1}{2}-\epsilon)}$. 
For any $N\in\N$ and $\bar\omega\in\bar\Omega$, denote by $(f_0^*,S_0^*)$ and by $(f_N^*(\bar\omega),S_N^*(\bar\omega))$ the optimal value and set of solutions to \ocp and to $\socp_N(\bar\omega)$, respectively. 
Then, $\bProb$-almost-surely, %
\revThird{$f_N^*(\bar\omega)\to f_0^*$ and $\mathbb{D}(S^*_N(\bar\omega),S_0^*)\to 0$ as $N\rightarrow\infty$.}
\end{theorem} 
\begin{proof} 
First, the function $f$ defined in \eqref{eq:fh_ocp} satisfies \A{1} \revThird{and \Aonetilde}. Indeed, let $f_1(u,\omega)=\int_{0}^{T} \ell(x_u(s,\omega),u(s))\, \dd s$ and $f_2(u,\omega)=\varphi(x_u(T,\omega))$ so that $f=f_1+f_2$. Next, we show that $f_1$ satisfies \A{1}. The analysis for $f_2$ is similar. 
We define the map $d:\Omega\to\R:\omega\mapsto \sup_{u\in\U}|f_1(u,\omega)|$ which satisfies $|f_1(u,\omega)|\leq d(\omega)$ for all $u\in\U$ $\Prob$-a.s. Also, $\E[d(\revThird{\omega})]<\infty$. 
Indeed, 
for any $t\in[0,T]$ and $u_0\in\U$,
$$
\sup_{u\in\U}|\ell(x_u(t),u(t))|\leq
|\ell(x_{u_0}(t),u_0(t))|+
\sup_{u\in\U}
|\ell(x_u(t),u(t))-\ell(x_{u_0}(t),u_0(t))|
$$
which is $\Prob$-almost-surely bounded by a finite constant thanks to \A{6} and \eqref{eq:xu_lip} (note that the constant $C_\alpha(\omega)$ is integrable and $\U$ is compact). 
Thus, $\E[d(\revThird{\omega})]<\infty$ \revThird{so $f$ satisfies \A{1}. $f$ also satisfies \Aonetilde thanks to \A{6} and \eqref{eq:xu_lip}.} 

Second, we show that $g$ and $h$ as defined in \eqref{eq:fh_ocp} satisfy \hyperref[A2]{(A2-3)} and \Afourtilde, respectively. Indeed, thanks to \A{6}, the maps $g$ and $h$ are bounded. $g$ and $h$ may not be Lipschitz continuous, but are $\alpha$-H\"older continuous for any $\alpha\in(0,\frac{1}{2})$. Indeed, thanks to \eqref{eq:xu_lip} and \A{6}, $\Prob$-almost-surely, for all $u,v\in\U$, 
$\|h(u,\omega)-h(v,\omega)\|\leq LC_\alpha(\omega)\|u-v\|_{L^2([0,T];U)}^\alpha$ 
for any $\alpha\in(0,\frac{1}{2})$ with $\E[C_\alpha(\omega)^2]<\infty$. 
To show that $g$ is also $\alpha$-H\"older continuous, %
 for any $j=1,\dots,q$ and $u,v\in\U$, we write
{\small
\begin{align*}
\revThird{\bigg|\sup_{t\in[0,T]}G_j(x_u(t,\omega))-\sup_{s\in[0,T]}G_j(x_v(s,\omega))\bigg|}
&\revThird{\leq}
\\[-1mm]
&\hspace{-40mm}\revThird{\leq
\sup_{t\in[0,T]}\left|G_j(x_u(t,\omega))-G_j(x_v(t,\omega))\right|
\leq
LC_\alpha(\omega)
\|u-v\|_{L^2([0,T];U)}^\alpha}
\end{align*}
}%
\revThird{by \A{6} and \eqref{eq:xu_lip}.} %
We apply Theorem \ref{thm:main:ineq} and the conclusion follows. 
\end{proof}

\section{Numerical results}\label{sec:numerical_results} 
\revSecond{We evaluate the proposed SAA approach on an optimization benchmark and on a stochastic optimal control problem. Python code to reproduce results is available at  \scalebox{1}{\url{https://github.com/StanfordASL/stochastic-prog}}.}

\subsection{\revSecond{Benchmark study}}
\revSecond{We consider 12 benchmark 
optimization problems with equality constraints from Hock and Schittkowski \cite{Hock1981} (problems 6, 27, 28, 42, 47-52, 61, and 79), as in the numerical study in \cite[Section 4]{KrklecJerinki2019}. These problems from \cite{Hock1981} take the form $\inf_{u\in\R^d}f_{\text{nom}}(u)$ such that $h_{\text{nom}}(u)=0$. To fit the problem formulation in \prob, we define $\U=\{u\in\R^d:\|u\|_\infty\leq 2\}$, 
$f(u,\omega)=f_{\text{nom}}(u)$, and
randomize the equality constraints by defining $h(u,\omega)=(h^1,\dots,h^n)(u,\omega)$ 
with $h^j(u,\omega)=h^j_{\text{nom}}((\omega_1u_1,\dots,\omega_du_d))+\omega_{d+j}$ for all $j=1,\dots,n$, 
where the $\omega_i$ are independent and Gaussian-distributed as $\omega_i\sim\mathcal{N}(1,5^2)$ for $i=1,\dots,d$ and  $\omega_{d+j}\sim\mathcal{N}(0,10^2)$ for $j=1,\dots,n$. 
The resulting problems \prob satisfy assumptions \hyperref[A1]{(\revThird{A1-3})}. We compute estimates of the optimal solutions  $u^\star$ of these stochastic programs \prob by solving deterministic reformulations of $\prob$ using \textrm{IPOPT} \cite{ipopt2006}, observing that the expected values $\E[h(u,\omega)]$ can be computed analytically for these problems \cite{KrklecJerinki2019}.  To apply the proposed SAA approach, we solve the sampled problems $\sprob_N(\bomega)$ using \textrm{IPOPT}, starting from the initial guesses $u^0$ from \cite{Hock1981}.}

\revSecond{We study the sensitivity of the proposed SAA approach to the sample size $N$ and to the relaxation constant $\delta_N$. We focus on the small sample size regime, which is practically relevant in many applications that face computational constraints or limited data. In Figure \ref{fig:results_benchmark}, \rev{for each choice of $(N,\delta_N)$,} we report the mean success rates (corresponding to successful numerical resolutions of the sampled problems $\sprob_N(\bomega)$ \rev{by \textrm{IPOPT}}) and \rev{median optimal value errors  $\frac{|f_N^\star(\bar\omega)-f_0^\star|}{1+|f_0^\star|}$ (corresponding to the suboptimality gap)} over $100$ runs of the SAA approach for \rev{all} program\rev{s} \prob. 
\rev{C}hoosing a larger relaxation constant  $\delta_N$ yields \rev{slightly} higher success rates \rev{and lower suboptimality gaps}, suggesting that the proposed SAA approach may perform better than the standard SAA approach (corresponding to $\delta_N=0$) for small sample sizes. 
Both performance metrics improve as the sample size increases, empirically validating the asymptotic optimality guarantees in Theorem \ref{thm:main:h_lipschitz_bounded}. 
}

\begin{figure}[htb]%
  \centering
\includegraphics[width=0.48\linewidth]{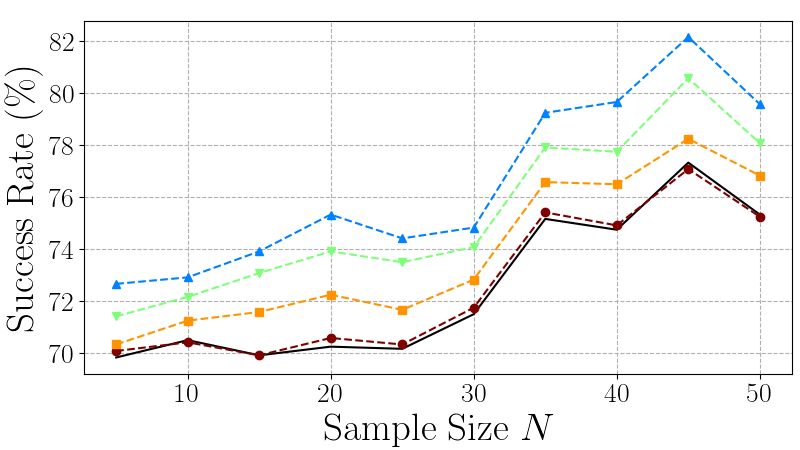}
\hspace{0.002\linewidth}
\includegraphics[width=0.48\linewidth]{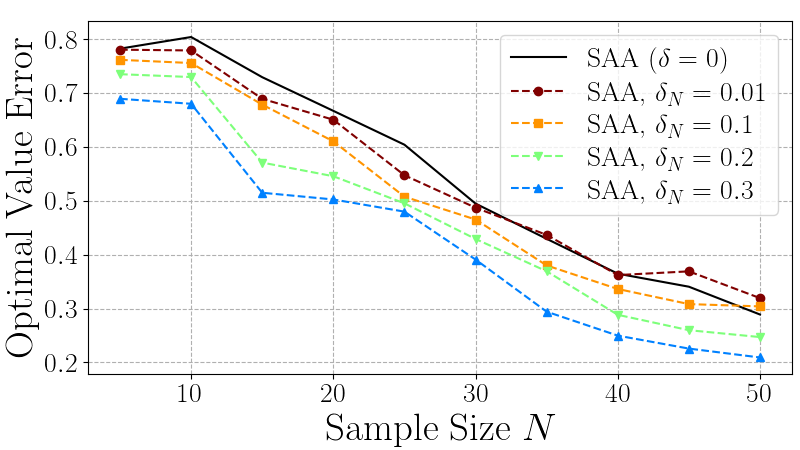}
  \caption{\revSecond{\rev{Benchmark study. Mean} success rates \rev{(representing successful numerical resolutions of the sampled problems)} and \rev{median optimal value errors} for different choices of $\delta_N$ and $N$.}
  \vspace{-6mm}}
\label{fig:results_benchmark}
\end{figure}

\subsection{\revSecond{Stochastic optimal control}}
We consider a Mars entry, descent, and landing (EDL) scenario where a spacecraft must safely deploy a rover at a landing site \cite{Steltzner2014}. Specifically, we consider the rocket-powered descent phase where the system must precisely reach a position above a landing site on Mars. 
The state of the system is given as $x=(r,v,m)\in\mathbb{R}^7$ where $r=(r_x,r_y,r_z)\in\R^3$ denotes the position,  $v\in\R^3$ the velocity, and $m\in\R$ the mass of the system. The control input is the thrust force $u=(u_x,u_y,u_z)\in\R^3$. We model the dynamics of the system with the SDE 
\begin{align}\label{eq:SDE_rocket}
\dd x_u(t)
&
=\begin{bmatrix}
v(t)\\
\frac{1}{m(t)}(\bar{T}u(t)-
\gamma |v|v)
- g
\\
-\revThird{\rho}\|u(t)\|
\end{bmatrix}\dt +\begin{bmatrix}
0\\
\frac{1}{m(t)}\text{diag}(\revThird{\eta_0} + %
\revThird{\eta_1 v^2)}
\\
\revThird{0} %
\end{bmatrix}\dd W_t
\end{align}
where $t\in[0,T]$ with 
$T=60\,\text{s}$, 
$\bar{T}=16\cdot10^3$, %
$g=(0,0,3.71)$, %
$\revThird{\rho}=8$, 
\revThird{$\gamma=1$, $\revThird{\eta_0}=10$, $\revThird{\eta_1}=0.2$,    
and}  
$\text{diag}(y_j)\in\R^{3\times 3}$ denotes the diagonal matrix with non-zero elements $y_j$ for any $y=(y_1,y_2,y_3)\in\R^3$. 
With this model, we account for %
disturbances from external wind gusts \revThird{and drag}. 
The \revThird{state} dependency of the diffusion coefficient \revThird{$\sigma(x)$} makes numerically solving this problem challenging. Prior work in the literature considers deterministic dynamics \cite{Acikmese2007,Szmuk2016,Leparoux2022} or a diffusion coefficient that only depends on the mass $m(t)$ \cite{Exarchos2019}.

The system starts at 
$x_u(0)=x_0=(300, 0, 1500, 5,  0, -75, 
1800)$ and must reach the goal state $x_g=(0,0,100,0,0,-10)$ in expectation, which we enforce by defining the final state equality constraint function
$H:\R^7\to\R^6:x\mapsto (r,v)-x_g$. 
We enforce a glide-slope inequality constraint at slope $\revThird{\psi}=\revThird{35}^\circ$ defined by $G:\R^7\to\R^4, 
\ 
x\mapsto(\tan(\revThird{\psi})r_x-r_z,
            \tan(\revThird{\psi})r_y-r_z,
            -\tan(\revThird{\psi})r_x-r_z,
            -\tan(\revThird{\psi})r_y-r_z).
$%

We minimize the control effort $\ell(x_u(t),u(t))=\|u(t)\|$ to reduce fuel consumption and the final dispersion $\varphi(x_u(T))=\|H(x_u(T))\|^2$. 
This choice for $\varphi$ minimizes the variance at the terminal state $x_u(T)$, since  $\E[\varphi(x_u(T))]=\text{Trace}(\text{Cov}(x_u(T)))$ if $\E[H(x_u(T))]=0$. 
Up to compositions with smooth cutoff functions, the functions $\varphi$ and $G,H$ satisfy \A{6}. 
The control space consists of bound and pointing constraints $U=\{u\in\R^3:\underline{u}\leq\|u(t)\|\leq\bar{u}, \|u(t)\|\cos(\theta)-u_z(t)\leq 0\}$ with min/maximum thrusts $(\underline{u},\bar{u})=(0.3,0.8)$ and pointing angle $\theta=45^\circ$. %
We consider piecewise-constant controls with $S=20$ switches. 
The complete %
formulation takes the form of \ocp. %

We discretize \eqref{eq:SDE_rocket} using an Euler-Maruyama scheme with $S$ timesteps and enforce the inequality constraint at these $S$ nodes. The analysis of the discretization error is beyond the scope of this paper and left for future work, see e.g. \cite{Kushner2001}. 
We %
\revThird{formulate $\socp(\bomega)$ using 
$N=20$ and $\delta_N=10^{-5}$. \revSecond{For this problem, we did not observe a significant influence of $\delta_N$ on the solution.}} \revThird{We solve $\socp(\bomega)$ via a standard direct method based on sequential convex programming. For further details, we refer to our open-source implementation. 
} 

We present results in Figure \ref{fig:results_rocket}. 
As a baseline, we solve the deterministic problem without uncertainty (i.e., \revThird{$\revThird{\eta_0}=\revThird{\eta_1}=0$}). %
The corresponding optimal strategy minimizes fuel consumption ($m(T)=\revThird{1592}\, \textrm{kg}$)\revThird{.} %
In contrast, when considering state-\revThird{dependent} %
uncertainty, one obtains a different trajectory that minimizes the magnitude of disturbances by reducing the \revThird{velocity %
magnitude} along the trajectory, albeit with larger fuel consumption ($m(T)=\revThird{1579}\, \textrm{kg}$). 
With $10^4$ Monte-Carlo simulations of \eqref{eq:SDE_rocket}, we observe that this solution %
reduces the variance of the error at the final time, i.e., $\E[\|H(x_u(T)\|^2]$ is minimized. \revThird{Indeed, the standard deviation of the final altitude $r_z(T)$ is only $28\,\textrm{m}$ for the proposed stochastic programming approach, versus $48\,\textrm{m}$ for the deterministic baseline.}  Thus, the trajectory computed by solving $\socp(\bomega)$ that considers uncertainty is potentially easier to track with a feedback controller.

\begin{figure}[t]%
  \centering
	\begin{minipage}{0.65\linewidth}
    \centering
\includegraphics[width=0.99\linewidth]{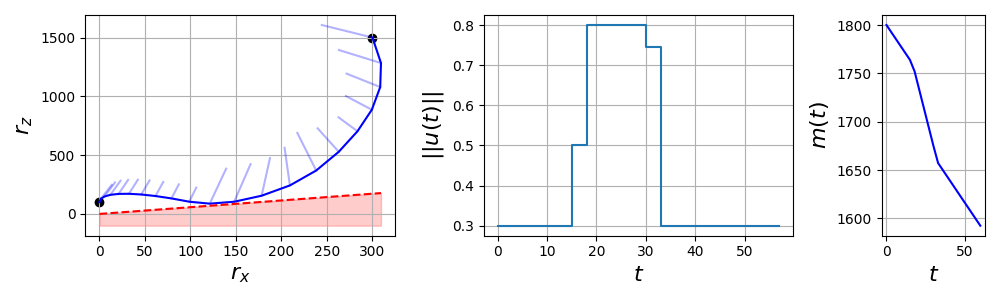}
\includegraphics[width=0.99\linewidth]{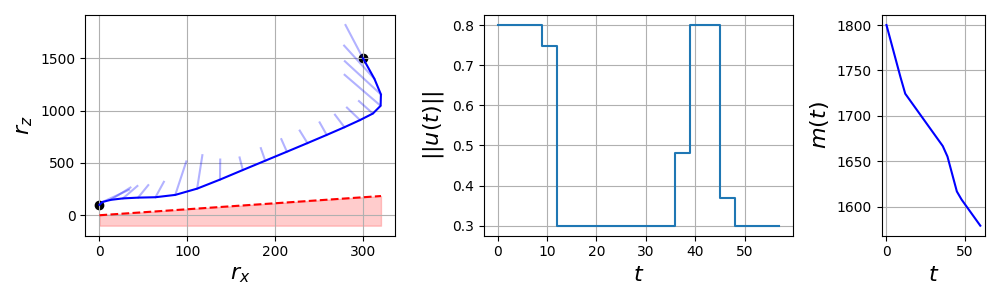}
	\end{minipage}%
	\hspace{1mm}
	\begin{minipage}{0.33\linewidth}
    \centering
\includegraphics[width=0.99\linewidth]{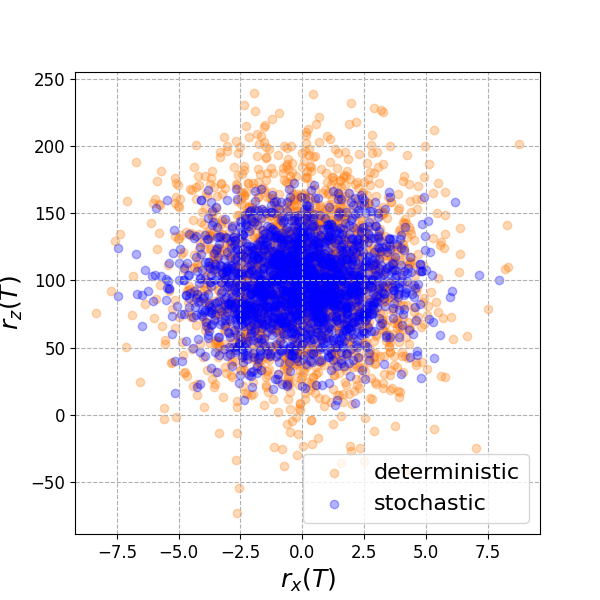}
	\end{minipage}%
  \caption{\revThird{Mars rocket-powered descent results. \textbf{Top-left}: trajectory obtained with a deterministic baseline that does not consider uncertainty. \textbf{Bottom-left}: trajectory obtained by solving $\socp_N(\bomega)$. \textbf{Right}: horizontal and vertical positions at the final time from Monte-Carlo simulations of \eqref{eq:SDE_rocket} with the controls from the deterministic baseline and from $\socp_N(\bomega)$.}}
\label{fig:results_rocket}
\end{figure}

\section{Conclusion}
\label{sec:conclusions}
We presented a modification to the SAA approach that makes the method applicable to a large class of stochastic programs with both expectation equality and inequality constraints. This modification consists of relaxing the equality constraints by a scalar $\delta_N$ that decreases to zero as the \revThird{sample size} $N$ increases. Under mild assumptions, we proved the consistency of the approach by reformulating the problem using random compact sets and by analyzing the approximation error with concentration inequalities. We applied the analysis to challenging stochastic optimal control problems and demonstrated the benefits of the approach \rev{on benchmark problems and} 
on a Mars powered-descent control problem.

\revFinal{Our analysis focused on the sets of optimal solutions to stochastic programs and to their sampled approximations. These sampled problems can be solved using off-the-shelf optimization tools tailored to specific problem classes. 
However, their numerical resolution can remain computationally expensive for large sample sizes. Our study may inform 
the design of 
specialized optimization algorithms to efficiently solve these reformulated sampled problems, e.g., by adapting $\delta_N$ and $N$ during the optimization process.  
Other interesting} 
directions of future work include 
deriving asymptotic convergence rates and finite-sample error bounds for the solution sets error $\mathbb{D}(S^\star_N(\bomega),S^\star_0)$, relaxing the assumption that $\U\cong\R^d$ to allow for optimization in more general function spaces, and considering risk functionals instead of expectations.

 \section*{Acknowledgements}
The National
Science Foundation (NSF) via the Cyber-Physical Systems (CPS) program (award \#1931815), the NASA University Leadership Initiative
(grant \#80NSSC20M0163), \revSecond{the French National Research Agency (ANR) (grant
ANR-22-CE48-0006),} and KACST provided funds to assist the authors with their research, but this article
solely reflects the opinions and conclusions of its authors and not any NSF, NASA, ANR, nor KACST entity.

\newpage
\appendix

\section{Sample space}\label{sec:product_space_socp}
\subsection{Product probability space}
Let $(\Omega,\G,\Prob)$ be a probability space. We define the product probability space $(\bOmega,\bar\G,\bProb)$ such that $\bOmega=\Omega^{\N}$ is the product space equipped with the product $\sigma$-algebra 
$\bar\G = \sigma(\{\pi^{-1}_i(B): i\in\N, B \in \G\})$, where $\pi_i:\bOmega\to\Omega : \bar\omega=(\omega^1,\dots)\mapsto\omega^i$,  
and $\bProb:\bar\G\to[0,1]$ is the product measure obtained via Kolmogorov's extension theorem \cite[Theorem 6.3]{LeGall2016} such that for any $N\in\N$,
$$
\bProb(A_N)=\underbrace{(\Prob\otimes\dots\otimes\Prob)}_{N \text{ times}}(A_N)
$$
for any $A_N\in\G^N$, with the slight abuse of notation $\bProb(A_N)\triangleq\bProb((\pi_1,\dots,\pi_N)^{-1}(A_N))$. 

Intuitively, each element $\omega^i$ of the \revThird{sample} $\bomega=(\omega^1,\dots)$ is interpreted as an i.i.d. \revThird{realization} of $\omega$, since $\bProb(\{\bomega\in\bOmega:\omega^i\in A\})=\bProb(\pi_i^{-1}(A))=\Prob(A)$ for any $A\in\G$.

\subsection{Sample space for stochastic optimal control}
In Section \ref{sec:stochastic_control}, $(\Omega,\G,\Prob)$ is the probability space that can be defined by setting 
$\Omega=C([0,T];\R^n)$, $\G=\B(\Omega)$ the Borel $\sigma$-algebra for the metric topology for the norm $\|x\|=\sup_{t\in[0,T]}\|x(t)\|$, 
$\Prob$ the Wiener measure \cite[Section 2.2]{LeGall2016}, and 
$W:\omega\mapsto \omega$ the canonical process that corresponds to Brownian motion with its canonical filtration $\F$. 
We denote by $(\bOmega,\bar\G,\bProb)$ the product probability space, 
by $\bomega=(\omega^1,\dots)\in\bOmega$ a \revThird{sample},  
and by 
$L^2_{\F}(\Omega;C([0,T];\R^n))$ the space of square-integrable $\F$-adapted stochastic processes with continuous sample paths taking values in $\R^n$.

For any $u\in\U$, the solution  to the SDE in \eqref{eq:SDE} satisfies $x_u\in L^2_{\F}(\Omega;C([0,T];\R^n))$. Our results rely on the measurability of the maps $f_N$, $g_N$, and $h_N$ defined from 
\eqref{eq:fh_ocp}, which depend on the measurability of the maps
\begin{align*}
x_u^i: ([0,T]\times\bOmega,\B([0,T])\otimes\bar\G)&\to (\R^n,\B(\R^n)),\ (t,\bomega)\mapsto x_u^i(t,\bomega)=x_u(t,\omega^i).
\end{align*}
To prove that $x_u^i$ is measurable, we define the measurable map 
\begin{align*}
t\otimes W^i:([0,T]\times\bOmega,\B([0,T])\otimes\bar\G)
&\to 
([0,T]\times\Omega,\B([0,T])\otimes\G), \ (s,\bomega)\mapsto (s,\omega^i).
\end{align*}
Then, the map $x_u^i$ can be written as the composition $x_u^i=x_u\circ (t\otimes W^i)$. 
Since the stochastic process $x_u:([0,T]\times\Omega,\B([0,T])\otimes\G)\to (\R^n,\B(\R^n)):(t,\omega)\mapsto x_u(t,\omega)$ is measurable, $x_u^i$ is measurable. By the continuity of the maps $\ell$, $\varphi$, $G$, and $H$ (to show the measurability of $g$, note that $x_u$ has continuous sample paths), we conclude that the maps $f_N$, $g_N$, and $h_N$ defined as in \eqref{eq:fN_gN} and \eqref{eq:g0_gN} from \eqref{eq:fh_ocp} are measurable.

\section{Convergence of random compact sets and proof of Lemma \ref{lemma:ULLN_convCN}}\label{apdx:random_sets}
\begin{wrapfigure}{R}{0.41\linewidth}
	\begin{minipage}{0.95\linewidth}
    \centering
\vspace{-5mm}
\includegraphics[width=0.35\linewidth]{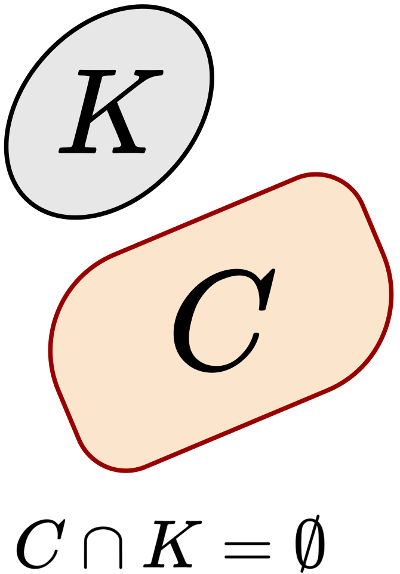}
\hspace{5mm}
\includegraphics[width=0.35\linewidth]{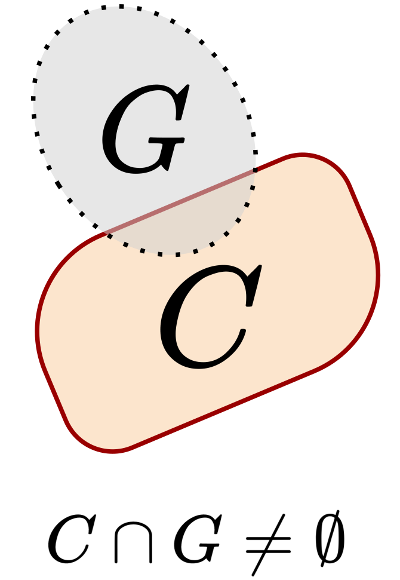}
\caption{Conditions of Theorem \ref{thm:conv_randSets_detLim}.}
\label{fig:thm:conv_randSets_detLim}
\vspace{-5mm}
	\end{minipage}%
\end{wrapfigure}
In this work, we use the myopic topology 

\noindent %
on the family $\K$ of compact subsets of $\R^d$ with its associated generated Borel $\sigma$-algebra $\B(\K)$. A random compact set is a map $C:\Omega\rightarrow\K$ that is measurable, i.e., $\{\omega\in\Omega : C(\omega)\in\mathscr{Y}\}\in\G$ for any $\mathscr{Y}\in\B(\K)$. 
We refer to  \cite{Matheron1975,Molchanov_BookTheoryOfRandomSets2017} for detailed treatments of random sets.

 The convergence of random compact sets can be characterized using the following result \cite[Proposition 1.7.23]{Molchanov_BookTheoryOfRandomSets2017} (see also \cite[Theorem 1-4-1]{Matheron1975} and \cite[Theorem D.3]{Molchanov_BookTheoryOfRandomSets2017}). 

\newpage
 \begin{theorem}[Convergence of random compact sets to a deterministic limit]\label{thm:conv_randSets_detLim} 
Let $(\Omega,\G,\Prob)$ be a probability space, 
$C\in\K$, and 
$\{C_N\}_{N=1}^\infty$ be a sequence of random compact sets. Assume that $\Prob$-almost-surely, there exists a set $K_0\in\K$ such that $C_N(\omega)\subset K_0$ for all $N\geq 1$, and that
\begin{itemize}[leftmargin=5.5mm]
\item 
For any $K\in\mathcal{K}$ such that $C\cap K=\emptyset$,
\begin{equation}\tag{{C1}}%
\hspace{-2mm}
\Prob(C_N\cap K \neq \emptyset \ \textrm{infinitely often}) 
= 
\Prob\left(\bigcap_{M=1}^\infty\bigcup_{N=M}^\infty
\{C_N\cap K \neq \emptyset\}\right) 
=
0.
\end{equation}
\item
For any open subset $G\subset\R^n$ such that $C\cap G \neq \emptyset$,
\begin{equation}\tag{{C2}}
\hspace{-2mm}
\Prob(C_N\cap G = \emptyset \ \textrm{infinitely often}) 
= 
\Prob\left(\bigcap_{M=1}^\infty\bigcup_{N=M}^\infty
\{C_N\cap G = \emptyset\}\right) 
= 
0.
\end{equation}
\item[] \hspace{-5.5mm}Then, $\Prob$-almost-surely, the sequence   $\{C_N\}_{N=1}^\infty$ converges to $C$ in the myopic topology.
\end{itemize}
\end{theorem}
We note that the random constraints sets $C_N(\bomega)$ defined in \eqref{eq:CN} are contained in the compact set $\U$. The requirement that $C_N(\bomega)\subset K_0\in\K$ for all $N\geq 1$ is necessary: the sets $C_N=\{0,N\}$ belong to $\K$ for all $N\in\N$ but do not converge in the myopic topology (although they converge to $\{0\}$ in the coarser Fell topology on the set of closed subsets of $\R^d$ \cite[Appendix D]{Molchanov_BookTheoryOfRandomSets2017}).

We use Theorem \ref{thm:conv_randSets_detLim} to prove Lemma \ref{lemma:ULLN_convCN}. 
\begin{proof}[Proof of Lemma \ref{lemma:ULLN_convCN}]
First, we prove that $\bProb$-a.s., $C_N(\bomega)\neq\emptyset$ for all $N\geq N(\bomega)$ with $N(\bomega)$ large enough.  
We define the ``infeasibility event'' 
\begin{align*}
B_N
=
\{\bomega\in\bOmega:\inf_{u\in\U}\|h_N(u,\bomega)\|>\delta_N\}
\subseteq
\{\bomega\in\bOmega:\|h_N(u^\star,\bomega)\|>\delta_N\}.
\end{align*}
Let $\bar{N}$ be large enough so that $\delta_N>\epsilon_N$ for all $N\geq \bar{N}$. Then, for $N\geq \bar{N}$, 
$$B_N\subseteq \{\bomega\in\bOmega:\|h_N(u^\star,\bomega)\|>\epsilon_N\}
=
\{\bomega\in\bOmega:\|h_N(u^\star,\bomega)-h_0(u^\star)\|>\epsilon_N\}.$$
Thus, $\bProb(B_N)\leq\revThird{\beta_N}$ for $N\geq \bar{N}$ by \eqref{eq:unif_gN_epsN_alphaN} and
$\sum_{N=1}^\infty \bProb(B_N)<\infty$. 
By the first Borel-Cantelli lemma, $\bProb$-a.s., $\prob_N(\bomega)$ is infeasible and $C_N(\bomega)=\emptyset$ only finitely many times.

Second, we prove that the following two conditions hold:
\begin{itemize}
\item For any $K\in\K$ such that $C_0\cap K=\emptyset$, 
\begin{equation}\tag{{C1}}\label{eq:C1}
\bProb\left(
C_N\cap K\neq\emptyset
\text{ infinitely often}
\right)=0.
\end{equation}
\item For any open subset $G\subset\U$ such that $C_0\cap G\neq\emptyset$,
\begin{equation}\tag{{C2}}\label{eq:C2}
\bProb\left(
C_N\cap G=\emptyset
\text{ infinitely often}
\right)=0.
\end{equation}
\end{itemize}

\begin{itemize}[leftmargin=5mm]
\item \textit{Proof of \eqref{eq:C1}:} Let $K\in\K$ be such that $C_0\cap K=\emptyset$.  

Then, there exists some $\bar\epsilon>0$ such that $\|h_0(u)\|\geq 2\bar\epsilon$ for all $u\in K$.   
Thus,  
\begin{align*}
0<2\bar\epsilon
\leq
\|h_0(u)\|
&\mathop{\leq}^\text{a.s.}
\|h_0-h_N(\cdot,\bomega)\|_\infty+\|h_N(u,\bomega)\|
\leq
\epsilon_N
+
\|h_N(u,\bomega)\|
\end{align*}
for all $u\in K$ with $\bProb$-probability (w.p.) at least $1-\revThird{\beta_N}$. Thus, w.p. at least $1-\revThird{\beta_N}$,
\begin{align}\label{eq:lemma:ULLN_convCN:1}
\inf_{u\in K}\|h_N(u,\bomega)\|
\geq 2\bar\epsilon-\epsilon_N
\geq \bar\epsilon
\end{align}
for $N$ large enough so that $\epsilon_N\leq\bar\epsilon$. 
To conclude, define the ``bad event'' 
\begin{align*}
B_N
=
\{\bomega\in\bOmega:C_N(\bomega)\cap K\neq\emptyset\}
=
\{\bomega\in\bOmega:\inf_{u\in K}\|h_N(u,\bomega)\|\leq\delta_N\}.
\end{align*}
Let $N$ be large enough so that $\epsilon_N\leq\bar\epsilon$ and $\delta_N\leq\bar\epsilon$. Then, $\bProb(B_N)\leq\revThird{\beta_N}$ by \eqref{eq:lemma:ULLN_convCN:1}. Thus, 
$
\sum_{N=1}^\infty \bProb(B_N)<\infty
$. 
We apply the first Borel-Cantelli lemma and conclude the proof of \eqref{eq:C1}.

\item \textit{Proof of \eqref{eq:C2}:} Let $G\subset\U$ be an open subset such that $C_0\cap G\neq\emptyset$. 

Thus, there exists some $u\in G$ such that $h_0(u)=0$. Then, 
with $\bProb$-probability at least $1-\revThird{\beta_N}$, 
for $N$ large enough so that $\epsilon_N\leq\delta_N$, 
\begin{align*}
\|h_N(u,\bomega)\|
&\leq
\|h_0-h_N(\cdot,\bomega)\|_\infty
+
\|h_0(u)\|
\leq
\epsilon_N
\leq
\delta_N
\end{align*}
so $u\in C_N(u,\bomega)$. This implies that $\bProb(C_N(\bomega)\cap G\neq\emptyset)\geq 1-\revThird{\beta_N}$. 
To conclude, define the ``bad event'' 
$B_N=
\{\bomega\in\bOmega:C_N(\bomega)\cap G =\emptyset\}$
so that 
$
\bProb(B_N)
=
1-\bProb(C_N(\bomega)\cap G \neq\emptyset)\leq \revThird{\beta_N}
$
for $N$ large enough. 
Then, 
$\sum_{N=1}^\infty\bProb(B_N)
=
\sum_{N=1}^\infty\left(
1-\bProb(C_N(\bomega)\cap G \neq\emptyset)\right)<\infty$. 
We apply the first Borel-Cantelli lemma and conclude the proof of \eqref{eq:C2}.
\end{itemize}
Using \eqref{eq:C1} and \eqref{eq:C2}, we apply Theorem \ref{thm:conv_randSets_detLim} (see \cite[Proposition 1.7.23]{Molchanov_BookTheoryOfRandomSets2017}) and 
conclude that the sequence of random compact sets $(C_N)_{N\in\N}$ converges to $C_0$ in the myopic topology $\bProb$-almost-surely. 
Since the myopic topology on $\K'$ is equivalent to the topology on $\K'$ generated by the Hausdorff metric on $\K'$ \cite{Matheron1975,Molchanov_BookTheoryOfRandomSets2017},  we obtain \eqref{eq:lemma:ULLN_convCN}.  
\end{proof}

\section{Asymptotic optimality of the SAA approach}\label{apdx:proof_thm:main:rand_sets} 

\subsection{Proof of Theorem \ref{thm:main:rand_sets}: the case with equality constraints}
We start with two lemmas. 
\begin{lemma}\label{lemma:there_is_a_subseq}
Let $(\bOmega,\bar\G,\bProb)$ be a probability space, 
$\U\subset\R^{d}$ and $C_0\subseteq\U$ be compact sets, 
$\{C_N\}_{N\in\N}$ be a sequence of random compact sets with $C_N:(\bOmega,\bar\G)\to(\K,\B(\K))$, and 
$\{u_N\}_{N\in\N}$ be a sequence of $\bar\G$-measurable random variables taking values in $\U$. Assume that:
\begin{itemize}
\item $\bProb$-almost-surely, $u_N(\bomega)\in C_N(\bomega)$ for all $N\in\N$. 
\item $\bProb$-almost-surely, 
$\bD(C_N(\bomega),C_0)\rightarrow 0$ 
as $N\rightarrow\infty$. 
\end{itemize}
Then, $\bProb$-almost-surely, there exists $u(\bomega)\in C_0$ and a subsequence $\{N_k(\bomega)\}_{k\in\N}$ such that $u_{N_k(\bomega)}(\bomega)\to u(\bomega)$ as $k\to\infty$. 
\end{lemma}
\begin{proof}
We fix $\bomega\in\bOmega$ in a set of $\bProb$-measure one. 
Since $u_N(\bomega)\in C_N(\bomega)$ for all $N\in\N$,  
$C_N(\bomega)\subset\U$ for all $N\in\N$, 
and $\U\subset\R^d$ is compact, there exists a subsequence $\{N_k(\bomega)\}_{k\in\N}$ and a limit point $u(\bomega)\in \U$  such that $u_{N_k(\bomega)}(\bomega)\to u(\bomega)$ as $k\to\infty$.  

To conclude, we show that $u(\bomega)\in C_0$. By the triangle inequality, for any $k\in\N$,
\begin{align*}
\bD(\{u(\bomega)\},C_0)
&\leq
\bD(\{u(\bomega)\},\{u_{N_k(\bomega)}(\bomega)\})
+
\bD(\{u_{N_k(\bomega)}(\bomega)\},C_{N_k}(\bomega))
+
\bD(C_{N_k}(\bomega),C_0)\\
&=
\|u(\bomega)-u_{N_k(\bomega)}(\bomega)\|
+
\bD(C_{N_k}(\bomega),C_0).
\end{align*}
Taking the limit as $k\to\infty$, we obtain $\bD(\{u(\bomega)\},C_0)=0$. Thus, $u(\bomega)\in C_0$. 
\end{proof}

\begin{lemma}\label{lemma:there_is_optimal_conv_seq}
With the notations introduced in Theorem \ref{thm:main:rand_sets},  assume that:
\begin{itemize}
\item $\bProb$-almost-surely, 
$\bD(C_0,C_N(\bomega))\rightarrow 0$ 
as $N\rightarrow\infty$. 
\end{itemize}
Let $u^\star\in S_0^\star$ be any optimal solution to \prob. 
Then, 
$\bProb$-almost-surely, there exists a sequence $\{u_N(\bomega)\}_{N\in\N}$ such that $u_N(\bomega)\in C_N(\bomega)$ for all $N\in\N$ and $u_N(\bomega)\to u^\star$ as $N\to\infty$. 
\end{lemma}

\begin{proof}
We fix $\bomega\in\bOmega$ in a set of $\bProb$-measure one and $u^\star\in S_0^\star\subseteq C_0$.  
For any $\epsilon>0$, there exists $N_\epsilon(\bomega)\in\N$ such that $\bD(C_0,C_N(\bomega))<\epsilon$ for all $N\geq N_\epsilon(\bomega)$. In particular, 
since $\mathbb{D}(C_0,\emptyset)=\infty$ by definition, 
$C_N(\bomega)$ is not empty for $N\geq N_\epsilon(\bomega)$.  %

Then, consider the sequence $\{u_N(\bomega)\}_{N\in\N}$ defined  for any $N\in\N$ as $u_N(\bomega)=\arg\inf_{u\in C_N(\bomega)}\|u-u^\star\|$. Note that  $u_N(\bomega)\in C_N(\bomega)$ for all $N\in\N$ since $C_N(\bomega)$ is closed. Since $u^\star\in C_0$,  for all $N\geq N_\epsilon(\bomega)$, 
$
\|u_N(\bomega)-u^\star\|
=
\bD(\{u^\star\},C_N(\bomega))
\leq
\bD(C_0,C_N(\bomega))
<
\epsilon.
$ 
Thus, the sequence $\{u_N(\bomega)\}_{N\in\N}$ satisfies  $\|u_N(\bomega) - u^\star\|<\epsilon$ for all $N\geq N_\epsilon(\bomega)$.  
As this statement holds for any $\epsilon>0$, this concludes the proof. 
\end{proof}

We observe that Lemma \ref{lemma:there_is_a_subseq} requires that $\bD(C_N(\bomega),C_0)\to 0$ and Lemma \ref{lemma:there_is_optimal_conv_seq} requires $\bD(C_0,C_N(\bomega))\to 0$ 
as $N\rightarrow\infty$. 
To ensure that both assumptions hold, we require that $\dHaus(C_N(\bomega),C_0)\rightarrow 0$ in Theorem \ref{thm:main:rand_sets}.

\begin{proof}[Proof of Theorem \ref{thm:main:rand_sets}]
We prove the two claims in \eqref{eq:thm:conv_sol_sets} sequentially. %
Throughout, we fix $\bomega\in\bOmega$ in a set of $\bProb$-measure one. 

\textbf{Part 1:  there exists a subsequence $\{N_k(\bomega)\}_{k\in\N}$ such that 
$f_{N_k(\bomega)}^\star(\bomega)\mathop{\longrightarrow}\limits_{k\to\infty}f_0^\star.$}
Consider any sequence $\{u_N^\star(\bomega)\}_{N\in\N}$ of solutions \revThird{to} $\sprob(\bomega)$, such that  $u_N^\star(\bomega)\in S_N^\star(\bomega)$ and $f_N(u_N^\star(\bomega),\bomega)=f_N^\star(\bomega)$. 
By Lemma \ref{lemma:there_is_a_subseq} (note that $u_N^\star(\bomega)\in C_N(\bomega)$ for all $N\in\N$ and $\lim_{N\to\infty}\dHaus(C_N(\bomega),C_0)=0$ 
by assumption in \eqref{eq:converges:CN}), there exists a limit point $u(\bomega)\in C_0$ and a subsequence $\{N_k(\bomega)\}_{k\in\N}$ such that 
\begin{equation}\label{eq:subseq:conv}
u_{N_k(\bomega)}^\star(\bomega)\to u(\bomega) \quad\text{as}\quad k\rightarrow\infty.
\end{equation} 
Next, we prove that for some $u^\star\in S_0^\star$, 
the subsequence 
$\{N_k(\bomega)\}_{k\in\N}$ is such that 
\begin{subequations}
\begin{gather}
\label{eq:liminf_larger}
\mathop{\lim\inf}\limits_{k\to\infty} f_{N_k(\bomega)}(u_{N_k(\bomega)}^\star(\bomega),\bomega)\geq f_0(u^\star)
\\
\label{eq:limsup_smaller}
\mathop{\lim\sup}\limits_{k\to\infty} f_{N_k(\bomega)}(u_{N_k(\bomega)}^\star(\bomega),\bomega)\leq f_0(u^\star).
\end{gather}
\end{subequations} 

To show \eqref{eq:liminf_larger}, we first note that $\lim_{k\to\infty}f_{N_k(\bomega)}(u_{N_k(\bomega)}^\star(\bomega),\bomega)= f_0(u(\bomega))$, since
\begin{align*}
&|f_{N_k(\bomega)}(u_{N_k(\bomega)}^\star(\bomega),\bomega)- f_0(u(\bomega))|
\\
&\hspace{1cm}\leq 
\underbrace{|f_{N_k(\bomega)}(u_{N_k(\bomega)}^\star(\bomega),\bomega)- f_0(u_{N_k(\bomega)}^\star(\bomega))|}_{\text{$\to 0$ as $k\to\infty$ by \eqref{eq:converges:fN:unif}}}
+
\underbrace{|f_0(u_{N_k(\bomega)}^\star(\bomega))- f_0(u(\bomega))|}_{\text{$\to 0$ as $k\to\infty$ ($f_0$ continuous and \eqref{eq:subseq:conv})}}.
\end{align*}
Also,   
$f_0(u(\bomega))\geq f_0(u^\star)$ for any $u^\star\in S_0^\star$, since $u(\bomega)\in C_0$.
Combining the last two results, we obtain 
\eqref{eq:liminf_larger}.

Next, we show \eqref{eq:limsup_smaller}. 
By Lemma \ref{lemma:there_is_optimal_conv_seq}, there exists a sequence $\{u_N(\bomega)\}_{N\in\N}$ such that $u_N(\bomega)\in C_N(\bomega)$ for all $N\in\N$ and $\lim_{N\to\infty}u_N(\bomega)=u^\star$ for any $u^\star\in S_0^\star$.

Thus, since $u_{N_k(\bomega)}^\star(\bomega)\in S_{N_k(\bomega)}^\star(\bomega)$ and $u_{N_k(\bomega)}(\bomega)\in C_{N_k(\bomega)}(\bomega)$ for all $k\in\N$, 
\begin{align*}
f_{N_k(\bomega)}(u_{N_k(\bomega)}^\star(\bomega),\bomega)
&\leq 
f_{N_k(\bomega)}(u_{N_k(\bomega)}(\bomega),\bomega)
-
f_0(u_{N_k(\bomega)}(\bomega))
+
f_0(u_{N_k(\bomega)}(\bomega)).
\end{align*}
By uniform convergence in \eqref{eq:converges:fN:unif}, $\lim_{k\to\infty}|f_{N_k(\bomega)}(u_{N_k(\bomega)}(\bomega),\bomega)
-
f_0(u_{N_k(\bomega)}(\bomega))|
=0$. 
Also, 
$\lim_{k\to\infty}f_0(u_{N_k(\bomega)}(\bomega))=f_0(u^\star)$ since 
$\lim_{N\to\infty}u_N(\bomega)=u^\star$ and $f_0$ is continuous.  
Combining the last results, we obtain
$$
f_{N_k(\bomega)}(u_{N_k(\bomega)}^\star(\bomega),\bomega)
\leq 
f_{N_k(\bomega)}(u_{N_k(\bomega)}(\bomega),\bomega)
\to
f_0(u^\star)
\quad\text{as}\quad
k\to\infty
$$
which shows \eqref{eq:limsup_smaller}. 

Together, \eqref{eq:liminf_larger} and \eqref{eq:limsup_smaller} imply\footnote{Indeed, $f_0(u^\star)\leq \mathop{\lim\inf}\limits_{k\to\infty} f_{N_k(\bomega)}(u_{N_k(\bomega)}^\star(\bomega),\bomega)\leq \mathop{\lim\sup}\limits_{k\to\infty} f_{N_k(\bomega)}(u_{N_k(\bomega)}^\star(\bomega),\bomega)\leq f_0(u^\star)$.} that 
$
\lim_{k\to\infty}f_{N_k(\bomega)}(u_{N_k(\bomega)}^\star(\bomega),\bomega)= f_0(u^\star)$. This concludes the proof of \eqref{eq:thm:conv_sol_sets:cost} in Theorem \ref{thm:main:rand_sets}.

\textbf{Part 2: the subsequence $\{N_k(\bomega)\}_{k\in\N}$ is such that 
$\bD(S_{N_k(\bomega)}^\star(\bomega),S_0^\star)\mathop{\longrightarrow}\limits_{k\to\infty}
0$.}

By contradiction, assume that for some $\epsilon>0$ and $\bomega\in\bOmega$,
$\|u_{N_k(\bomega)}^\star(\bomega)-u^\star\|\geq\epsilon$ for all $u^\star\in S_0^\star$ and all $k\in\N$. 
This implies that
$\|u(\bomega)-u^\star\|\geq\epsilon$ for all $u^\star\in S_0^\star$, 
since 
$
\epsilon
\leq
\|u_{N_k(\bomega)}^\star(\bomega)-u^\star\|
\leq
\|u_{N_k(\bomega)}^\star(\bomega)-u(\bomega)\|
+
\|u(\bomega)-u^\star\|$ for all $k\in\N$, so that $\|u(\bomega)-u^\star\|=\lim_{k\to\infty}\left(\|u_{N_k(\bomega)}^\star(\bomega)-u(\bomega)\|
+
\|u(\bomega)-u^\star\|\right)\geq \epsilon$. 

Thus, $
f_0(u(\bomega))
>
f_0(u^\star)
$. 
Since $f_0$ is continuous and $u_{N_k(\bomega)}^\star(\bomega)\to u(\bomega)$, we obtain 
\begin{equation}\label{eq:cost_inequality_contradiction}
\lim_{k\rightarrow\infty}
f_0(u_{N_k(\bomega)}^\star(\bomega))
=
f_0(u(\bomega))
>
f_0(u^\star).
\end{equation}
However, 
$
|f_0(u(\bomega))-f_0(u^\star)|
\leq 
|f_0(u(\bomega))-f_0(u_{N_k(\bomega)}^\star(\bomega))|
+
|f_{N_k(\bomega)}(u_{N_k(\bomega)}^\star(\bomega),\bomega)- f_0(u^\star)|
+
|f_0(u_{N_k(\bomega)}^\star(\bomega))-f_{N_k(\bomega)}(u_{N_k(\bomega)}^\star(\bomega),\bomega)|
$ 
which goes to zero as $k\to\infty$ due to \eqref{eq:converges:fN:unif} and Part 1 (i.e., \eqref{eq:thm:conv_sol_sets:cost}). 
Thus, $f_0(u(\bomega))=f_0(u^\star)$, which contradicts \eqref{eq:cost_inequality_contradiction}. 
 This concludes the proof of \eqref{eq:thm:conv_sol_sets:set} in Theorem \ref{thm:main:rand_sets}.

\revThird{\textbf{Part 3: if $f$ also satisfies \Aonetilde, then the full sequence converges.}} 
\revThird{Under this additional assumption, $\Prob$-almost-surely, there exists $\alpha,L(\omega)>0$ such that 
$|f(u_1,\omega)-f(u_2,\omega)|\leq L(\omega)\|u_1-u_2\|^\alpha$ for all $u_1,u_2\in\U$. %
In particular, $|f_N(u_1,\bomega)-f_N(u_2,\bomega)|\leq L(\bomega)\|u_1-u_2\|^\alpha$. For conciseness, we drop dependencies on $(\omega,\bomega)$. }

\revThird{It suffices to prove that $\lim_{N\to\infty}\bD(S_N^\star,S_0^\star)=0$. By contradiction, there exists $\bar\epsilon>0$ such that for all $M\in\N$, there is $N_M\geq M$ with \smash{$\sup_{x\in S_{N_M}^\star}\Inf_{y\in S_0^\star}\|x-y\|\geq\bar\epsilon$}. Thus, there exists a sequence $\{u_M\}_{M\in\N}$ with $u_M\in S_{N_M}^\star$ and $\inf_{y\in S_0^\star}\|u_M-y\|\geq\bar\epsilon$ for all $M\in\N$. }

\revThird{Since $u_M\in\U$ for all $M\in\N$ and $\U$ is compact, there exists a subsequence $\{u_{M_k}\}_{k\in\N}$ such that $u_{M_k}\to\bar{u}$ as $k\to\infty$ for some $\bar{u}\in C_0$ (thanks to \eqref{eq:converges:CN}) with $\lim_{k\to\infty}|f_{N_{M_k}}(u_{M_k})-f_0(\bar{u})|= 0$ (thanks to \eqref{eq:converges:fN:unif}). Since $\inf_{y\in S_0^\star}\|u_M-y\|\geq\bar\epsilon$ for all $M\in\N$, we have $\bar{u}\in C_0\setminus S_0^\star$, so
$$
\inf_{u\in C_0}f_0(u)<f_0(\bar{u})=\lim_{k\to\infty}f_{N_{M_k}}(u_{M_k})=
\lim_{k\to\infty}\inf_{u\in C_{N_{M_k}}}f_{N_{M_k}}(u).
$$
To conclude, it suffices to prove that 
$\lim_{N\to\infty}\inf_{u\in C_N}f_N(u)\leq\inf_{u\in C_0}f_0(u)$: %
\begin{align*}
\inf_{u\in C_N}f_N(u)-\inf_{u\in C_0}f_0(u)
&=
\inf_{u\in C_N}f_N(u)-f_0(u^\star)%
\quad\text{(for some $u^\star\in C_0$)}
\\
&=
\inf_{u\in C_N}f_N(u)-f_N(u^\star)+f_N(u^\star)-f_0(u^\star)
\\
&\leq
L\inf_{u\in C_N}\|u-u^\star\|^\alpha
+
\sup_{u\in\U}|f_N(u)-f_0(u)|
\\
&\leq
Ld_H(C_N,C_0)^\alpha
+
\sup_{u\in\U}|f_N(u)-f_0(u)|
\end{align*}
which converges to zero as $N\to\infty$ thanks to \eqref{eq:converges:CN} and \eqref{eq:converges:fN:unif}.}
\end{proof}

\subsection{Proof of Theorem \ref{thm:main:ineq}: the case with equality and inequality constraints}
We first extend Lemmas \ref{lemma:there_is_a_subseq} and \ref{lemma:there_is_optimal_conv_seq} to account for inequality constraints.
\begin{lemma}\label{lemma:there_is_a_subseq:ineq}
Let $(\Omega,\G,\Prob)$ be a probability space, 
$(\bOmega,\bar\G,\bProb)$ be the product probability space, 
$\U\subset\R^{d}$ and $C_0\subseteq\U$ be compact sets, 
$\{C_N\}_{N\in\N}$ be a sequence of random compact sets $C_N:(\bOmega,\bar\G)\to(\K,\B(\K))$, and 
$\{u_N\}_{N\in\N}$ be a sequence of random variables $u_N:(\bOmega,\bar\G)\to(\U,\B(\U))$. 
Let $g:\U\times\Omega\to\R^q$ be a  Carath\'eodory function and define $g_0$ and $g_N$ as in \eqref{eq:g0_gN}.  
Assume that $\bProb$-almost-surely, 
\begin{itemize}
\item  $u_N(\bomega)\in C_N(\bomega)$ for all $N\in\N$. 
\item $g_N(u_N(\bomega),\bomega)\leq 0$ for all $N\in\N$. 
\item $\bD(C_N(\bomega),C_0)\rightarrow 0$ 
as $N\rightarrow\infty$. 
\item $g$ satisfies \Afoura. 
\end{itemize}
Then, $\bProb$-almost-surely, there exists $u(\bomega)\in C_0$ with $g_0(u(\bomega))\leq 0$ and a subsequence $\{N_k(\bomega)\}_{k\in\N}$ such that $u_{N_k(\bomega)}(\bomega)\to u(\bomega)$ as $k\to\infty$. 
\end{lemma}
\begin{proof}
First, $\bProb$-almost-surely,  $\lim_{N\to\infty}\sup_{u\in\U}\|g_N(u,\bar\omega)-g_0(u)\|=0$ from \Afoura.

Consider the point $u(\bomega)\in C_0$ and the subsequence $\{N_k(\bomega)\}_{k\in\N}$ from Lemma \ref{lemma:there_is_a_subseq} such that $u_{N_k(\bomega)}(\bomega)\to u(\bomega)$ as $k\to\infty$. Then, $g_0(u(\bomega))\leq 0$. Indeed, 
\begin{align*}
\|g_N(u_N(\bomega),\bomega)){-}g_0(u(\bomega))\|
&\leq
\|g_N(u_N(\bomega),\bomega)){-}g_0(u_N(\bomega)))\|{+}\|g_0(u_N(\bomega)){-}g_0(u(\bomega))\|.
\end{align*}
Thus, by the uniform convergence from \Afoura and the continuity of $g_0$, by passing to the subsequence $N_k(\bomega)$, we obtain that  $g_0(u(\bomega))\leq 0$.
\end{proof}

\begin{lemma}\label{lemma:there_is_optimal_conv_seq:ineq}
With the notations used in Theorem \ref{thm:main:ineq}, assume that:
\begin{itemize}
\item $h$ satisfies \hyperref[A2]{(A2-3)}.
\item $g$ satisfies \Afour. 
\end{itemize}
Let  $u^\star\in S_0^\star$ denote the optimal solution from \Afourb. 
Then, $\bProb$-almost-surely, for $\tilde{N}_\epsilon(\bomega)\in\N$ large enough, there exists a sequence $\{u_N(\bomega)\}_{N\in\N}$ such that $u_N(\bomega)\in C_N(\bomega)$  and  $g_N(u_N(\bomega),\bomega)\leq 0$ for all $N\geq \tilde{N}_\epsilon(\bomega)$, and $u_N(\bomega)\to u^\star$ as $N\to\infty$. 
\end{lemma}
\begin{proof}
Given $\bomega\in\bOmega$, we consider the sequence $\{u_N(\bomega)\}_{N\in\N}$ defined  as 
$$
u_N(\bomega)=\mathop{\arg\inf}%
_{u\in (C_N(\bomega)\cap \{u\in\U:g_N(u,\bomega)\leq 0\})}\|u-u^\star\|.
$$ 
$\bProb$-almost-surely, there exists $N(\bomega)\in\N$ such that $u_N(\bomega)\neq\emptyset$ for all $N\geq N(\bomega)$ (i.e., the intersection of $C_N(\bomega)$ with $\{u\in\U:g_N(u,\bomega)\leq 0\}$ is not empty). 

To prove this claim, 
we let $\epsilon\in(0,\frac{1}{2})$ and define $\revThird{\beta_N}$ and $\epsilon_N$ as in \eqref{eq:eps_N_alpha_N} and $\delta_N$ as in \eqref{eq:delta_N}. Then, 
we consider the sequence $\{v_{\epsilon_N}\}_{N\in\N}$ that satisfies $g_0(v_{\epsilon_N})<0$ and $\|v_{\epsilon_N}-u^\star\|^\alpha<{\epsilon_N}$, which is well-defined thanks to \Afourb. 

Thanks to \Afoura, $\bProb$-a.s., $\lim_{N\to\infty}\sup_{u\in\U}\|g_N(u,\bomega)-g_0(u)\|=0$, so $\|g_N(v_{\epsilon_N},\bomega)-g_0(v_{\epsilon_N})\|\to 0$ as $N\to\infty$. Thus, $\bProb$-a.s., $g_N(v_{\epsilon_N},\bomega)\leq0$ for all $N>\hat{N}(\bomega)$ with $\hat{N}(\bomega)\in\N$ large enough. 
Also, since $h_0(u^\star)=0$ and thanks to 
\hyperref[A2]{(A2-3)}, Proposition \ref{prop:concent:unif_bounded}, and Jensen's inequality, with $\bProb$-probability at least $1-\revThird{\beta_N}$, 
$\|h_N(v_{\epsilon_N},\bomega)\| 
\leq
\|h_N(v_{\epsilon_N},\bomega)-h_0(v_{\epsilon_N})\|+\|h_0(v_{\epsilon_N})-h_0(u^\star)\| 
\leq
(1+\E[M(\revThird{\omega})])\epsilon_N
\leq
\delta_N$  
for all $N\geq\bar{N}$ with $\bar{N}\in\N$ large enough so that $(1+\E[M(\revThird{\omega})])\epsilon_N\leq\delta_N$ for all $N\geq\bar{N}$ (note that $\epsilon_N/\delta_N\to 0$ as $N\to\infty$).  
Next, we define the infeasibility event $B_N=\{\bomega\in\bOmega:\|h_N(v_{\epsilon_N},\bomega)\|>\delta_N\}$. By the above, $\bProb(B_N)\leq\revThird{\beta_N}$ for all $N\geq\bar{N}$. By the first Borel-Cantelli Lemma, $\bProb$-a.s., $B_N$ only happens finitely many times. 

Thus, $\bProb$-a.s.,  $u_N(\bomega)\neq\emptyset$ for all $N\geq N(\bomega)$ for $N(\bomega)\in\N$ large enough. 
Finally, $\|u_N(\bomega)-u^\star\|\leq\|v_{\epsilon_N}(\bomega)-u^\star\|\leq\epsilon_N$ for $N\geq N(\bomega)$ with $\lim_{N\to\infty}\epsilon_N=0$. 
\end{proof}

\begin{proof}[Proof of Theorem \ref{thm:main:ineq} if $g$ satisfies \Afour] %
By \revThird{\A{1} and \hyperref[A2]{(A2-3)}},  $
\sup_{u\in\U}|f_N(u,\bar\omega)-f_0(u)|\to 0$ and $\dHaus(C_N(\bar\omega),C_0)\to 0$ as $N\to\infty$ \revThird{$\bProb$-a.s.}. 
Also, by \hyperref[A2]{(A2-4)} and Lemma  \ref{lemma:there_is_optimal_conv_seq:ineq}, $\bProb$-almost-surely, $\sprob_N(\bomega)$ is feasible for all $N\geq N(\bomega)$ with $N(\bomega)$ large enough.

The conclusion of the proof follows from the proof of Theorem \ref{thm:main:rand_sets} by replacing the use of Lemmas \ref{lemma:there_is_a_subseq} and \ref{lemma:there_is_optimal_conv_seq} with Lemmas \ref{lemma:there_is_a_subseq:ineq} and \ref{lemma:there_is_optimal_conv_seq:ineq}. 
\end{proof}

\section{Proving Propositions \ref{prop:concent:F_alphaHolder}-\ref{prop:concent:unif_bounded} with concentration inequalities}
\label{appdx:prop:concent:Fcentered_lip} 
Section \ref{sec:concentration} leverages concentration inequalities, namely McDiarmid's inequality and Dudley's entropy integral. 
In this section, we prove Propositions \ref{prop:concent:F_alphaHolder}-\ref{prop:concent:unif_bounded}. 

\subsection{Rademacher variables, Rademacher complexity, and symmetrization} To derive a uniform law of large numbers of the form of \eqref{eq:unif_gN_epsN_alphaN} for the function class $\cH$ in \eqref{eq:fclass:H}, we define 
the independent Rademacher variables\footnote{To define $(\bar\Omega,\bar\G,\bar\Prob)$ accounting for these random variables, one can start from the product space in Section \ref{sec:problem_formulation} and redefine $\bar\Omega\gets\bar\Omega\times\{-1,1\}^\N$ and similarly for $\bar\G$ and $\bProb$, see Appendix \ref{sec:product_space_socp}.} $\epsilon_i:\bar\Omega\to\{-1,1\}$ so that $\bar\Prob(\epsilon_i=1)=\bar\Prob(\epsilon_i=-1)=\frac{1}{2}$ (see, e.g., \cite{ShalevShwartz2009,wainwright_2019}) and the Rademacher complexity %
\begin{equation}
\label{eq:Rademacher}
R_N(\cH)
=
\bar\E\left[
\sup_{h_u\in\cH}\left\|\frac{1}{N}\sum_{i=1}^N\epsilon_i h_u(\omega^i)\right\|
\right].
\end{equation}
By a standard symmetrization argument (see e.g. \cite[(4.17)-(4.18)]{wainwright_2019}), %
\begin{equation}\label{eq:E_unif_RN}
\bar\E\left(
\sup_{h_u\in\cH}\left\|\frac{1}{N}\sum_{i=1}^Nh_u(\omega^i)-\E[h_u(\revThird{\omega})]\right\|
\right)
\leq 
2R_N(\cH).
\end{equation}
Combining \eqref{eq:E_unif_RN} with Markov's inequality gives \eqref{eq:Markov_and_RN} in Remark \ref{remark:unif_Markov}.

\subsection{Dudley's entropy integral}      Dudley's entropy integral \cite[Theorem 5.22]{wainwright_2019} is a concentration inequality exploiting the tight tails of sub-Gaussian processes and covering arguments via chaining. We recall basic concepts before stating the result.
 
\begin{definition}[Sub-Gaussian process]
Let $(\Omega,\G,\Prob)$ be a probability space and $(\Theta,\rho)$ be a metric space. An $\R$-valued stochastic process $\{X_\theta:\theta\in\Theta\}$ is sub-Gaussian with respect to the metric $\rho$ on $\Theta$ if 
$
\E\left[
\exp(\lambda(X_{\theta_1}-X_{\theta_2}))
\right]
\leq
\exp\left(
\lambda^2\rho(\theta_1,\theta_2)^2/2
\right)$ for all $\lambda\in\R$ and all $\theta_1,\theta_2\in\Theta$. 
\end{definition}
Sub-Gaussian random variables satisfy the following properties.
\begin{enumerate}[leftmargin=7mm]
\item If $X_1,\dots,X_N$ are independent $\sigma^2$-sub-Gaussian, $\sum_{i=1}^NX_i$ is $(N\sigma^2)$-sub-Gaussian.
\item If $X$ is $\sigma^2$-sub-Gaussian and $a>0$, then $aX$ is $(a^2\sigma^2)$-sub-Gaussian.
\item If $\underline{x}\leq X\leq \bar{x}$ almost surely and $\E[X]=0$, then $X$ is $\left((\bar{x}-\underline{x})^2/4\right)$-sub-Gaussian.
\end{enumerate} 

\begin{definition}[Covering] 
Let $(\Theta,\rho)$ be a metric space and $\epsilon>0$. 
A collection $\{\theta_1,\dots,\theta_N\}$ is called an $\epsilon$-cover if 
$
\min_{i}\rho(\theta,\theta_i)\leq\epsilon$ for all $\theta\in\Theta$. 
  The $\epsilon$-covering number of $(\Theta,\rho)$ is defined as
$ 
N(\Theta,\rho,\epsilon)=\inf\left\{
N\in\N:\exists\text{ an }\epsilon\text{-cover }\{\theta_1,\dots,\theta_N\}
\right\}
$.
\end{definition}
The covering number of any compact set $\Theta\subset\R^d$ satisfies 
\begin{equation}\label{eq:covering_2norm}
N(\Theta,\|\cdot\|,\epsilon)\leq \left(1+D/\epsilon\right)^d
\end{equation}
with $D=2\sup_{\theta\in\Theta}\|\theta\|$, see e.g. \cite[(5.9)]{wainwright_2019}. %
Then, $N(\Theta,\|\cdot\|^\alpha,\epsilon^\alpha)\leq \big(1+D/\epsilon\big)^d$ for any $\alpha\in(0,1]$\footnote{\label{footnote:metric_alpha}Note that $\|\cdot\|^\alpha$ is a metric for any $\alpha\in(0,1]$. Indeed, (1) $\|u-v\|^\alpha=0\iff u,v=0$, (2) $\|u-v\|^\alpha=\|v-u\|^\alpha$, and (3) $\|u-w\|^\alpha\leq \|u-v\|^\alpha+\|v-w\|^\alpha$ since $(x+y)^\alpha\leq x^\alpha+y^\alpha$ for all $x,y\geq 0$ (proof: for $y$ fixed, consider $f(x)=(x+y)^\alpha-x^\alpha-y^\alpha$. Note that $f(0)=0$ and $f'(x)\leq 0$ for all $x,y\geq 0$ so $f(x)$ is nonincreasing. Thus, $f(x)\leq 0$ for all $x\geq 0$. }. Indeed, the $\epsilon$-cover $\{\theta_1,\dots,\theta_N\}$ for the Euclidean metric $\|\cdot\|$ that satisfies $\min\|\theta-\theta_i\|\leq\epsilon$ for all $\theta\in\Theta$ also satisfies $\min\|\theta-\theta_i\|^\alpha\leq\epsilon^\alpha$ for all $\theta\in\Theta$.

\begin{theorem}[Dudley's entropy integral]\label{thm:dudley} 
Let $(\Omega,\G,\Prob)$ be a probability space, 
$(\Theta,\rho)$ be a compact metric space,  
$\{X_\theta:\theta\in\Theta\}$ be an $\R$-valued sub-Gaussian process with respect to the metric $\rho$ on $\Theta$,  
$D=\sup_{\theta,\theta'\in\Theta}\rho(\theta,\theta')$, and 
$C=32$. Then, 
$$
\E\left[
\sup_{\theta\in\Theta}X_\theta
\right]
\leq
C\int_0^D\sqrt{\log(N(\Theta,\rho,u))}\dd u.
$$
\end{theorem} 
Theorem \ref{thm:dudley} allows obtaining bounds on $\E\left[\sup_{\theta\in\Theta}\left|X_\theta\right|\right]$ for stochastic processes $X_\theta$ that are symmetric (i.e., $X_\theta$ has the same distribution as $-X_\theta$):
\begin{align}\label{eq:supXtheta_abs}
\E\left[\sup_{\theta\in\Theta}
|X_\theta|\right]
&\leq
2\E\left[\sup_{\theta\in\Theta}X_\theta\right] 
+
\inf_{\tilde\theta\in\Theta}\E\left[|X_{\tilde\theta}|\right].
\end{align}
Indeed, by the triangle inequality\footnote{By the triangle inequality, we have $\sup_\theta|X_\theta|=\sup_\theta\inf_{\tilde\theta}|X_\theta|\leq\sup_\theta\inf_{\tilde\theta}(|X_\theta-X_{\tilde\theta}|+|X_{\tilde\theta}|)
\leq
\sup_\theta\inf_{\tilde\theta}(\sup_{\tilde\theta}(|X_\theta-X_{\tilde\theta}|)+|X_{\tilde\theta}|)
\leq
\sup_{\theta,\tilde\theta}|X_\theta-X_{\tilde\theta}|+\inf_{\tilde\theta}|X_{\tilde\theta}|$.}, if $X_\theta$ is a symmetric process, 
$\E[\sup_{\theta\in\Theta}
|X_\theta|]
\leq
\E[\sup_{\theta,\tilde\theta\in\Theta}|X_\theta-X_{\tilde\theta}|]+
\inf_{\tilde\theta\in\Theta}\E[|X_{\tilde\theta}|]
=
2\E[\sup_{\theta\in\Theta}X_\theta] 
+
\inf_{\tilde\theta\in\Theta}\E[|X_{\tilde\theta}|]$, 
see \cite[Lemma 2.2.1]{Talagrand2014}. 
The first term in \eqref{eq:supXtheta_abs} can be bounded using Theorem \ref{thm:dudley}.

\subsection{Proof of Proposition \ref{prop:concent:F_alphaHolder}}\label{sec:proof:prop:concent:F_alphaHolder}
We first prove a concentration result for centered smooth function classes. We then use this result to prove  Proposition \ref{prop:concent:F_alphaHolder}. 
\begin{proposition}[Concentration for centered $\alpha$-H\"older function class]\label{prop:concent:Fcentered_lip:alpha}
With the notations and assumptions of Proposition \ref{prop:concent:F_alphaHolder}, define the centered function class $\tilde\cH=\{h(u,\cdot)-h(u_0,\cdot)\}_{u\in\U}$. Then, 
$$
\bar\E\left[\sup_{h_u\in\tilde\cH}\left\|\frac{1}{N}\sum_{i=1}^Nh_u(\omega^i)-\E[h_u(\revThird{\omega})]\right\|\right]
\leq 
\frac{1}{\sqrt{N}}
\frac{8CD^{\frac{\alpha+1}{2}}d^{\frac{1}{2}}n^{\frac{3}{2}}}{\alpha^{\frac{1}{2}}}
\sqrt{\E\left[M^2\right]}.
$$
\end{proposition}
\begin{proof}
First, we denote $h(u,\omega)=(h^1(u,\omega),\dots,h^n(u,\omega))\in\R^n$ and define the dimension-wise centered function classes $\tilde\cH^j=\{h^j(u,\cdot)-h^j(u_0,\cdot):\Omega\to\R,
\ u\in\U\}$ with $j=1,\dots,n$. 
We note that any $h_u^j\in\tilde\cH^j$ satisfies \A{2} with $\alpha$-H\"older constant 
\begin{equation}\label{eq:Mj_sqrtn_M}
M_j(\omega)=\sqrt{n}M(\omega)
\end{equation}
since $\|a\|_1\leq\sqrt{n}\|a\|$ for any $a\in\R^n$ and 
$
|h^j(u_1,\omega)-h^j(u_2,\omega)|
\leq 
\|h(u_1,\omega)-h(u_2,\omega)\|_1
\leq
\sqrt{n}
\|h(u_1,\omega)-h(u_2,\omega)\|.
$ 
Since $\|a\|\leq\|a\|_1$ for any $a\in\R^n$,  
\begin{equation}\label{eq:dimensionwise}
\bar\E\left[\sup_{h_u\in\tilde\cH}\left\|\frac{1}{N}\sum_{i=1}^Nh_u(\omega^i)-\E[h_u]\right\|\right]
\leq
\sum_{j=1}^n
\bar\E\left[\sup_{h_u^j\in\tilde\cH^j}\left|\frac{1}{N}\sum_{i=1}^Nh_u^j(\omega^i)-\E[h_u^j]\right|\right]
.\end{equation}
Next, we fix $j=1,\dots,n$ and bound the right-hand-side term above dimension-wise. 
We define the $\R$-valued centered process $Z_u^j$ and the empirical expectation $\E_N[M_j^2]$:
$$
Z_u^j\triangleq \frac{1}{\sqrt{N}}\sum_{i=1}^N\epsilon_i(h^j(u,\omega^i)-h^j(u_0,\omega^i)),
\qquad
\E_N[M_j^2]\triangleq \frac{1}{N}\sum_{i=1}^NM_j(\omega^i)^2.
$$
When conditioned on $\bomega$, the process $\{Z_u^j:u\in\U\}$ is $\rho^2$-sub-gaussian for $\rho(u_1,u_2)^2=\E_N[M_j^2]\|u_1-u_2\|^{2\alpha}$. 
Indeed, the Rademacher variables $\epsilon_i$ take values in $\{-1,1\}$ with equal probability, so they are $1$-sub-Gaussian. Thus, $\epsilon_i a$ is $a^2$-sub-Gaussian for any $a\geq 0$, and since the maps $u\mapsto h^j(u,\omega^i)$ satisfy \A{2}, 
{\small
\begin{align*}
\bar\E\left[\exp(\lambda(Z_{u_1}^j-Z_{u_2}^j))
\ \big|\ 
\bomega \right]
&=
\bar\E\bigg[\exp\bigg(\frac{\lambda}{\sqrt{N}}\sum_{i=1}^N\epsilon_i(h^j(u_1,\omega^i)-h^j(u_2,\omega^i)\bigg)
\ \bigg|\ 
\bomega \bigg]
\\
&\hspace{-35mm}\leq
\exp\left(\frac{\lambda^2}{2N}\sum_{i=1}^N
(h^j(u_1,\omega^i)-h^j(u_2,\omega^i))^2
\right)
\mathop{\leq}^{\text{\A{2}}}
\exp\left(\frac{\lambda^2}{2}\left(\frac{1}{N}\sum_{i=1}^NM_j(\omega^i)^2\right)\|u_1-u_2\|^{2\alpha}
\right)
\end{align*}
}%
where $\bar\E\left[A\, |\, \bomega \right]$ denotes the conditional expectation of the random variable $A$ given $\bomega=(\omega^1,\dots)$, i.e., the expected value of $A$ over the Rademacher variables $\epsilon_i$. %

In addition, $Z_u^j$ is a symmetric process, i.e., $Z_u^j$ has the same distribution as $-Z_u^j$ when conditioned on $\bomega$. %
Thus, if we take the expectation over the Rademacher variables $\epsilon_i$, since $Z^j_{u_0}=0$, we obtain  
\begin{align}\label{eq:absolute_symmetric}
\bar\E\left[\sup_{u\in\U}
|Z_u^j|
\,\Big|\,\bomega
\right]
&\mathop{\leq}^{\eqref{eq:supXtheta_abs}}
2\bar\E\left[\sup_{u\in\U}Z_u^j
\,\Big|\,\bomega
\right]
+
\inf_{\tilde{u}\in\U}\bar\E\left[|Z_{\tilde{u}}^j|
\,\Big|\,\bomega
\right]
=
2\bar\E\left[\sup_{u\in\U}Z_u^j
\,\Big|\,\bomega
\right].
\end{align}
Moreover, $Z_u^j/\sqrt{\E_N[M_j^2]}$ is $\|u_1-u_2\|^{2\alpha}$-sub-Gaussian when conditioned on $\bomega$. Thus, we apply  Dudley's entropy integral (Theorem \ref{thm:dudley}) and obtain 
{\small
\begin{align*}
\bar\E\left[\sup_{u\in\U}|Z_u^j|\ \Big|\ \bomega\right]
&\mathop{\leq}^{\eqref{eq:absolute_symmetric}}
2\bar\E\left[\sup_{u\in\U}Z_u^j\ \Big|\ \bomega\right]
=
2\sqrt{\E_N[M_j^2]}
\bar\E\left[\sup_{u\in\U}\frac{Z_u^j}{\sqrt{\E_N[M_j^2]}}\ \Bigg|\ \bomega\right] 
\\
&\hspace{-18mm}\leq
2C\E_N[M_j^2]^{\frac{1}{2}}
\int_0^D\sqrt{\log N(\U,\|\cdot\|^{\alpha},u)}\dd u
\hspace{3mm}\mathop{\leq}^{\eqref{eq:covering_2norm}}
2C\E_N[M_j^2]^{\frac{1}{2}}
\int_0^D\sqrt{
    d\log\left(1{+}\frac{D}{u^{\frac{1}{\alpha}}}\right)
    }\dd u
\\
&\hspace{-18mm}\leq
2C\E_N[M_j^2]^{\frac{1}{2}}
\int_0^D
\sqrt{
d\log
	\left(\left(1{+}
		\frac{D^\alpha}{u}
		\right)^{\frac{1}{\alpha}}\right)
	}\dd u
=
2C\E_N[M_j^2]^{\frac{1}{2}}
\int_0^D
\sqrt{d\frac{1}{\alpha}\log
	\left(
	1{+}
		    \frac{D^\alpha}{u}
	\right)
	}\dd u
\\
&\hspace{-18mm}\leq
2C\E_N[M_j^2]^{\frac{1}{2}}
\int_0^D\sqrt{
    \frac{dD^\alpha}{\alpha u}}\dd u
\hspace{22mm}=
\frac{4CD^{\frac{\alpha+1}{2}}d^{\frac{1}{2}}}{\alpha^{\frac{1}{2}}}
\E_N[M_j^2]^{\frac{1}{2}}
\end{align*}
}%
since $\left(1+\frac{D}{u^{1/\alpha}}\right)^\alpha\leq 1+\left(\frac{D}{u^{1/\alpha}}\right)^\alpha$ (see footnote \ref{footnote:metric_alpha}) and $\log(1+t)\leq t$ for $t>-1$. 
Finally, %
by Jensen's inequality (J) and since $M_j(\omega)=\sqrt{n}M(\omega)$ \eqref{eq:Mj_sqrtn_M},  
{\small
\begin{align*}
\bar\E\left[\sup_{h_u\in\tilde\cH}\left\|\frac{1}{N}\sum_{i=1}^Nh_u(\omega^i)-\E[h_u(\revThird{\omega})]\right\|\right]
&
\mathop{\leq}^{\eqref{eq:dimensionwise}}
\sum_{j=1}^n
\bar\E\left[\sup_{h_u^j\in\tilde\cH^j}\left|\frac{1}{N}\sum_{i=1}^Nh_u^j(\omega^i)-\E[h_u^j]\right|\right]
\\
&\hspace{-40mm}\mathop{\leq}^{\eqref{eq:E_unif_RN}}
\sum_{j=1}^n
2R_N(\tilde\cH^j)
=
\sum_{j=1}^n
\frac{2}{\sqrt{N}}
\bar\E\left[\sup_{u\in\U}|Z_u^j| \right] 
=
\sum_{j=1}^n
\frac{2}{\sqrt{N}}
\bar\E\left[\bar\E\left[\sup_{u\in\U}|Z_u^j| \ \Big|\ \bomega \right]
\right]
\\
&\hspace{-39mm}\leq
\sum_{j=1}^n
\frac{2}{\sqrt{N}}
\frac{4CD^{\frac{\alpha+1}{2}}d^{\frac{1}{2}}}{\alpha^{\frac{1}{2}}}
\bar\E\left[\sqrt{\E_N[M_j^2]}\right]
\mathop\leq^{\text{(J)}}
\frac{1}{\sqrt{N}}
\sum_{j=1}^n
\frac{8CD^{\frac{\alpha+1}{2}}d^{\frac{1}{2}}}{\alpha^{\frac{1}{2}}}
\sqrt{\bar\E\left[\E_N[M_j^2]\right]}
\\
&\hspace{-40mm}\mathop{=}^{\eqref{eq:Mj_sqrtn_M}}
\frac{1}{\sqrt{N}}
\frac{8CD^{\frac{\alpha+1}{2}}d^{\frac{1}{2}}n^{\frac{3}{2}}}{\alpha^{\frac{1}{2}}}
\sqrt{\E\left[M^2\right]}
\end{align*} 
}%
which concludes the proof.
\end{proof}

Proposition \ref{prop:concent:Fcentered_lip:alpha} treats the centered function class $\tilde\cH=\{h(u,\cdot)-h(u_0,\cdot)\}_{u\in\U}$. 
Proposition \ref{prop:concent:F_alphaHolder} follows by extending this result to the function class $\cH$  %
in \eqref{eq:fclass:H}.

\begin{proof}[Proof of Proposition \ref{prop:concent:F_alphaHolder}]
By the triangle inequality,
{\small
\begin{align*}
\left\|\frac{1}{N}\sum_{i=1}^Nh(u,\omega^i)-\E[h(u,\revThird{\omega})]\right\|
&\leq
\left\|\frac{1}{N}\sum_{i=1}^N(h(u,\omega^i){-}h(u_0,\omega^i))-\E[(h(u,\revThird{\omega}){-}h(u_0,\revThird{\omega}))]\right\|
\\
&\qquad
+
\left\|\frac{1}{N}\sum_{i=1}^Nh(u_0,\omega^i)-\E[h(u_0,\revThird{\omega})]\right\|.
\end{align*}
}%
With the centered function class $\tilde\cH=\{h(u,\cdot)-h(u_0,\cdot)\}_{u\in\U}$ defined in Proposition  \ref{prop:concent:Fcentered_lip:alpha} and the function class $\cH=\{h(u,\cdot)\}_{u\in\U}$ defined 
in \eqref{eq:fclass:H}, 
we obtain
{\small
\begin{align*}
\sup_{h_u\in\cH}\left\|\frac{1}{N}\sum_{i=1}^Nh_u(\omega^i)-\E[h_u]\right\|
&\leq
\sup_{\tilde{h}_u\in\tilde\cH}\left\|\frac{1}{N}\sum_{i=1}^N\tilde{h}_u(\omega^i)-\E[\tilde{h}_u]\right\|
+
\left\|\frac{1}{N}\sum_{i=1}^Nh_{u_0}(\omega^i)-\E[h_{u_0}]\right\|.
\end{align*}
}%
where we write $h_{u_0}=h(u_0,\cdot)\in\cH$. 

Proposition  \ref{prop:concent:Fcentered_lip:alpha} bounds the expected value of the first term.  

To bound the expected value  of the second, by Jensen's inequality,
{\small
\begin{align*}
\bar\E\left[\left\|
\frac{1}{N}\sum_{i=1}^Nh(u_0,\omega^i)-\E[h(u_0,\revThird{\omega})]
\right\|\right]
&\leq
\bar\E\left[\left\|
\frac{1}{N}\sum_{i=1}^Nh(u_0,\omega^i)-\E[h(u_0,\revThird{\omega})]
\right\|^2\right]^{\frac{1}{2}}
\\
&\hspace{-48mm}=
\text{Trace}\left(
\text{Cov}\left(\frac{1}{N}\sum_{i=1}^Nh(u_0,\omega^i)\right)
\right)^{\frac{1}{2}}
=
\text{Trace}\left(
\frac{1}{N}\text{Cov}\left(h(u_0,\cdot)\right)
\right)^{\frac{1}{2}}
=
\frac{1}{\sqrt{N}}
\text{Trace}\left(\Sigma_0\right)^{\frac{1}{2}}.
\end{align*}
}%
Combining Proposition \ref{prop:concent:Fcentered_lip:alpha} and this bound gives the result.
\end{proof}

\subsection{Proof of Proposition \ref{prop:concent:unif_bounded}}\label{sec:proof:prop:concent:unif_bounded}
We use Corollary \ref{cor:concent:unif_bounded} and Proposition \ref{prop:concent:F_alphaHolder}. 
\begin{proof}[Proof of Proposition \ref{prop:concent:unif_bounded}]
By Corollary \ref{cor:concent:unif_bounded}, for any $\delta>0$, 
with $\bProb$-probability less than $\exp\left(-\frac{N\delta^2}{2\bar{h}^2}\right)$, 
$$
\sup_{u\in\U}\left\|\frac{1}{N}\sum_{i=1}^Nh(u,\omega^i)-\E[h(u,\revThird{\omega})]\right\|
\geq
\bar\E\left[\sup_{u\in\U}\left\|\frac{1}{N}\sum_{i=1}^Nh(u,\omega^i)-\E[h(u,\revThird{\omega})]\right\|\right]
+\delta.
$$
Also, note that the  boundedness of $h(u_0,\omega)$ in \A{3} implies that $\Sigma_0$ is finite. Then, 
$$
\bar\E\left[
\sup_{u\in\U}\left\|\frac{1}{N}\sum_{i=1}^Nh(u,\omega^i)-\E[h(u,\revThird{\omega})]\right\|\right]
\leq 
\frac{1}{\sqrt{N}}\tilde{C}
$$ 
by Proposition \ref{prop:concent:F_alphaHolder}. 
Thus, combining the previous two results,
\begin{equation}\label{eq:prop:concent:unif_bounded:as_fn_of_delta}
\bar\Prob\left(
\sup_{u\in\U}\left\|\frac{1}{N}\sum_{i=1}^Nh(u,\omega^i)-\E[h(u,\revThird{\omega})]\right\|
\geq
\frac{\tilde{C}}{\sqrt{N}}+\delta
\right)
\leq 
\exp\left(-\frac{N\delta^2}{2\bar{h}^2}\right).
\end{equation}
Let $\epsilon>0$ and $\delta=\tilde{C}N^{-\frac{1}{2}+\frac{\epsilon}{2}}$. Then, $\left(\tilde{C}N^{-\frac{1}{2}}+\tilde{C}N^{-\frac{1}{2}+\frac{\epsilon}{2}}\right)\leq 2\tilde{C}N^{-\frac{1}{2}+\frac{\epsilon}{2}}$, so
\begin{equation*}
\bar\Prob\left(
\sup_{u\in\U}\left\|\frac{1}{N}\sum_{i=1}^Nh(u,\omega^i)-\E[h(u,\revThird{\omega})]\right\|
\geq
2\tilde{C}N^{-\frac{1}{2}+\frac{\epsilon}{2}}
\right)
\leq 
\exp\left(-\frac{\tilde{C}^2}{2\bar{h}^2}N^\epsilon\right).
\end{equation*}
Let $\epsilon_N=2\tilde{C}N^{-\frac{1}{2}+\frac{\epsilon}{2}}$ and $\revThird{\beta_N}=\exp\left(-\frac{\tilde{C}^2}{2\bar{h}^2}N^\epsilon\right)$ and the conclusion follows.
\end{proof}

\begin{proof}[Proof of Corollary \ref{cor:prop:concent:unif_bounded}]
\revThird{The result follows from the proof of Proposition \ref{prop:concent:unif_bounded} by defining $\delta=(2\bar{h}^2\log(1/\beta)/N)^{\frac{1}{2}}$ in \eqref{eq:prop:concent:unif_bounded:as_fn_of_delta}.
}
\end{proof}
\section{Continuity of state trajectories with respect to control input}\label{sec:xu_continuous}
Let $u,v\in \U$ and denote by $x^u$ and $x^v$ the associated trajectories that satisfy the SDE \eqref{eq:SDE}. For any $t\in[0,T]$, we write $\Delta b^{u,v}_t=b(x_t^u,u_t)-b(x_t^v,v_t)$ and  $\Delta\sigma^{u,v}_t=\sigma(x_t^u,u_t)-\sigma(x_t^v,v_t)$. Then, for any $t\in[0,T]$ and $p\geq 2$, we have
{\small
\begin{align*}
\E\left[\sup_{0\leq s\leq t}\|x_s^u-x_s^v\|^p\right]
&=
\E\bigg[
\sup_{0\leq s\leq t}
\bigg\|
\int_0^s
\Delta b^{u,v}_r\dd r
+
\int_0^s
\Delta\sigma^{u,v}_r\dd W_r
\bigg\|^p
\bigg]
\\
&\leq 
C_p\bigg(
\E\bigg[\sup_{0\leq s\leq t}
	\bigg\|
	\int_0^s
	\Delta b^{u,v}_r\dd r
	\bigg\|^p
\bigg]
+
\E\bigg[\sup_{0\leq s\leq t}
	\bigg\|
	\int_0^s
	\Delta\sigma^{u,v}_r\dd W_r
	\bigg\|^p
\bigg]
\bigg)
\end{align*}
}%
for some constant $C_p$. 
Then, by H\"older's inequality,
{\small
\begin{align*}
\E\bigg[\sup_{0\leq s\leq t}
	\bigg\|
	\int_0^s
	\Delta b^{u,v}_r\dd r
	\bigg\|^p
\bigg] 
\leq
\E\bigg[
	\bigg(
	\int_0^t
	\|\Delta b^{u,v}_r\|\dd r
	\bigg)^p
\bigg] 
\leq
t^{p-1}
\E\bigg[
	\int_0^t
	\|\Delta b^{u,v}_r\|^p\dd r
\bigg] 
\end{align*}
}%
and by the Burkholder-Davis-Gundy \cite[Theorem 5.16]{LeGall2016} and H\"older inequalities,
{\small
\begin{align*}
\E\bigg[\sup_{0\leq s\leq t}
	\bigg\|
	\int_0^s
	\Delta\sigma^{u,v}_r\dd W_r
	\bigg\|^p
\bigg]
&\leq
C_2
\E\bigg[\left\langle
	\int_0^\cdot
	\Delta\sigma^{u,v}_r\dd W_r,
	\int_0^\cdot
	\Delta\sigma^{u,v}_r\dd W_r
	\right\rangle_t^{p/2}
\bigg]
\\
&=
C_2
\E\bigg[\left(
	\int_0^t
	\|\Delta\sigma^{u,v}_r\|^2\dd r\right)^{p/2}
\bigg]
=C_2t^{\frac{p}{2}-1}
\E\bigg[
	\int_0^t
	\|\Delta\sigma^{u,v}_r\|^p\dd r
\bigg].
\end{align*}
}%
Thus, since $t\in[0,T]$, we obtain for some constants $C_{p,T}$ and $C_{p,T,K}$ that
{\small
\begin{align*}
\E\left[\sup_{0\leq s\leq t}\|x_s^u-x_s^v\|^p\right]
&\leq
C_{p,T}\bigg(
\E\bigg[
	\int_0^t
	\|\Delta b^{u,v}_r\|^p\dd r
\bigg] 
+
\E\bigg[
	\int_0^t
	\|\Delta\sigma^{u,v}_r\|^p\dd r
\bigg]
\bigg)
\\
&\leq
C_{p,T,K}\bigg(
\E\bigg[
	\int_0^t
	\|x_r^u-x_r^v\|^p\dd r
\bigg] 
+
	\int_0^t
	\|u_r-v_r\|^p\dd r
\bigg)
\\
&\leq
C_{p,T,K}\bigg(
\E\bigg[
	\int_0^t
	\sup_{0\leq r' \leq r}
	\|x_{r'}^u-x_{r'}^v\|^p\dd r
\bigg] 
+
	\int_0^t
	\|u_r-v_r\|^p\dd r
\bigg)
\end{align*}
}%
where the second inequality follows from \ALipschitz. 
Thus, by Gr\"onwall's inequality,
{\small
\begin{align*}
\E\left[\sup_{0\leq s\leq t}\|x_s^u-x_s^v\|^p\right]
&\leq 
C_{p,T,K}
\left(
\int_0^t\|u_r-v_r\|^p\dd r
\right)
\end{align*}
}%
for a new constant $C_{p,T,K}$ and all $t\in[0,T]$. %
Thus, for $p=2$ and $t=T$,
{\small
\begin{align*}
\E\left[\sup_{0\leq s\leq T}\|x_s^u-x_s^v\|^p\right]
&\leq 
C_{T,K}
\left(\bigg(
\int_0^T\|u_s-v_s\|^2\dd s
\bigg)^\frac{1}{2}
\right)^{1+\beta}
\end{align*}
}%
for $\beta=1$. 
Then, since $\U$ can be identified with a (compact) subset of $\R^{Sm}$, we apply Kolmogorov's Lemma \cite[Theorem 2.9]{LeGall2016} and obtain that there exists a modification $\tilde{x}^u$ of $x^u$ whose sample paths are $\alpha$-H\"older continuous for every $\alpha\in(0,\frac{\beta}{p})=(0,\frac{1}{2})$. Thus, for every $\alpha\in(0,\frac{1}{2})$,  $\Prob$-almost-surely, there exists a finite constant $C_\alpha(\omega)$ with
{\small
\begin{align*}
\sup_{0\leq t\leq T}\|\tilde{x}_t^u(\omega)-\tilde{x}_t^v(\omega)\|
&\leq 
C_\alpha(\omega)
\left(\bigg(
\int_0^T\|u_s-v_s\|^2\dd s
\bigg)^{\frac{1}{2}}\right)^\alpha.
\end{align*}
}%
This shows that up to a modification,  $\Prob$-almost-surely, the map 
$
x_\cdot(\cdot,\omega):\U\to C([0,T];\R^n):\, u\mapsto x_u(\cdot,\omega)
$ 
is continuous, with the metric topology on $C([0,T];\R^n)$ induced by the norm $\|x\|_\infty=\sup_{0\leq t\leq T}\|x(t)\|$ and the metric topology on $\U\cong U^S\subset\R^{Sm}$ for $\|u\|=\big(\int_0^Tu^2(s)\dd s\big)^{\frac{1}{2}}$.

To show that $\E[C_\alpha^2(\revThird{\omega})]<\infty$, by adopting the notation and results introduced in the proof of \cite[Theorem 2.9]{LeGall2016} (the variables $(q,\epsilon)$ in \cite{LeGall2016} and in the proof below evaluate to $(2,1)$ in this work), it suffices to prove that
$
\mathbb{E}[ K^q_{\alpha}(\revThird{\omega}) ] < \infty
$, 
where
{\small
$$
K_{\alpha}(\omega) \triangleq \underset{n \in \mathbb{N}}{\sup} \left( \underset{1 \le i \le 2^n}{\max} \ 2^{\alpha n} d\left( X_{\frac{i-1}{2^n}}(\omega) , X_{\frac{i}{2^n}}(\omega) \right) \right) .
$$
}%
To prove that $\E[K_\alpha^q(\revThird{\omega})]<\infty$, thanks to the monotone convergence theorem, we write
{\footnotesize
\begin{align*}
    \mathbb{E}[ K^q_{\alpha} ] &= \mathbb{E}\left[ \left( \underset{n \in \mathbb{N}}{\sup} \left( \underset{1 \le i \le 2^n}{\max} \ 2^{\alpha n} d\left( X_{\frac{i-1}{2^n}} , X_{\frac{i}{2^n}} \right) \right) \right)^q \right] 
    \le \mathbb{E}\left[ \sum_{n \in \mathbb{N}} \left( \underset{1 \le i \le 2^n}{\max} \ 2^{\alpha n} d\left( X_{\frac{i-1}{2^n}} , X_{\frac{i}{2^n}} \right) \right)^q \right] 
    \\
    &= \sum_{n \in \mathbb{N}} \mathbb{E}\left[ \left( \underset{1 \le i \le 2^n}{\max} \ 2^{\alpha n} d\left( X_{\frac{i-1}{2^n}} , X_{\frac{i}{2^n}} \right) \right)^q \right] 
    \le \sum_{n \in \mathbb{N}} 2^{q \alpha n} \sum^{2^n}_{i=1} \mathbb{E}\left[ d\left( X_{\frac{i-1}{2^n}} , X_{\frac{i}{2^n}} \right)^q \right] \\
    &\le \sum_{n \in \mathbb{N}} 2^{q \alpha n} \sum^{2^n}_{i=1} \left( \frac{1}{2^n} \right)^{1 + \varepsilon} 
    = \sum_{n \in \mathbb{N}} \left( \frac{1}{2^{\varepsilon - q \alpha}} \right)^n = \frac{2^{\varepsilon - q \alpha}}{2^{\varepsilon - q \alpha} - 1} < \infty
\end{align*}
}%
given that $\varepsilon - q \alpha > 0$.  
In this work, $\alpha\in(0,\frac{1}{2})$, so we conclude that $\E[C_\alpha^2(\revThird{\omega})]<\infty$.

\renewcommand{\baselinestretch}{0.95}

\bibliographystyle{siamplain}
\bibliography{ASL_papers,main}

\newcommand{\noopsort}[1]{} \newcommand{\printfirst}[2]{#1}
  \newcommand{\singleletter}[1]{#1} \newcommand{\switchargs}[2]{#2#1}
\begin{thebibliography}{10}

\bibitem{Acikmese2007}
{\sc B.~Acikmese and S.~R. Ploen}, {\em Convex programming approach to powered
  descent guidance for {Mars} landing}, {AIAA Journal of Guidance, Control, and
  Dynamics}, 30 (2007), pp.~1353--1366.

\bibitem{Anisimov2000}
{\sc V.~V. Anisimov and G.~C. Pflug}, {\em {Z-Theorems}: Limits of stochastic
  equations}, Bernoulli, 6 (2000), pp.~917--938.

\bibitem{Aubin1990}
{\sc J.-P. Aubin and H.~Frankowska}, {\em Set-Valued Analysis}, {Springer
  Science \& Business Media}, 1990.

\bibitem{Bartlett2003}
{\sc P.~L. Bartlett and S.~Mendelson}, {\em {Rademacher} and {Gaussian}
  complexities: Risk bounds and structural results}, {Journal of Machine
  Learning Research}, 3 (2003), pp.~463--482.

\bibitem{berret2020}
{\sc B.~Berret and F.~Jean}, {\em Efficient computation of optimal open-loop
  controls for stochastic systems}, {Automatica}, 115 (2020).

\bibitem{blackmore2011chance}
{\sc L.~Blackmore, M.~Ono, and B.~C. Williams}, {\em Chance-constrained optimal
  path planning with obstacles}, {IEEE Transactions on Robotics}, 27 (2011),
  pp.~1080--1094.

\bibitem{BonalliBonnet2022}
{\sc R.~Bonalli and B.~Bonnet}, {\em First-order {Pontryagin} maximum principle
  for risk-averse stochastic optimal control problems}.
\newblock Available at \url{https://arxiv.org/abs/2204.03036}, 2022.

\bibitem{BonalliLewESAIM2021}
{\sc R.~Bonalli, T.~Lew, and M.~Pavone}, {\em Sequential convex programming for
  non-linear stochastic optimal control},  (2021).
\newblock Available at \url{https://arxiv.org/abs/2009.05182}.

\bibitem{Bonnans2019}
{\sc J.~F. Bonnans}, {\em Convex and Stochastic Optimization}, {Springer},
  2019.

\bibitem{Campi2009}
{\sc M.~C. Campi, S.~Garatti, and M.~Prandini}, {\em The scenario approach for
  systems and control design}, {Annual Reviews in Control}, 33 (2009),
  pp.~149--157.

\bibitem{Exarchos2019}
{\sc I.~Exarchos, E.~A. Theodorou, and P.~Tsiotras}, {\em Optimal thrust
  profile for planetary soft landing under stochastic disturbances}, {AIAA
  Journal of Guidance, Control, and Dynamics}, 42 (2019), pp.~209--216.

\bibitem{Frankowska2019}
{\sc H.~Frankowska, H.~Zhang, and X.~Zhang}, {\em Necessary optimality
  conditions for local minimizers of stochastic optimal control problems with
  state constraints}, {Transactions of the American Mathematical Society}, 372
  (2019), pp.~1289--1331.

\bibitem{Hock1981}
{\sc W.~Hock and K.~Schittkowski}, {\em Test Examples for Nonlinear Programming
  Codes}, {Springer Berlin Heidelberg}, 1981.

\bibitem{Homem2014}
{\sc T.~Homem-de Mello and G.~Bayraksan}, {\em Monte {Carlo} sampling-based
  methods for stochastic optimization}, {Surveys in Operations Research and
  Management Science}, 19 (2014), pp.~56--85.

\bibitem{Hu2020siam}
{\sc Y.~Hu, X.~Chen, and N.~He}, {\em Sample complexity of sample average
  approximation for conditional stochastic optimization}, {SIAM Journal on
  Optimization}, 30 (2020), pp.~2103--2133.

\bibitem{KrklecJerinki2019}
{\sc N.~K. Jerinki{\'{c}} and A.~Ro{\v{z}}njik}, {\em Penalty variable sample
  size method for solving optimization problems with equality constraints in a
  form of mathematical expectation}, Numerical Algorithms, 83 (2019),
  pp.~701--718.

\bibitem{Kleywegt2002}
{\sc A.~J. Kleywegt, A.~Shapiro, and T.~Homem-de Mello}, {\em The sample
  average approximation method for stochastic discrete optimization}, {SIAM
  Journal on Optimization}, 12 (2002), pp.~479--502.

\bibitem{Koltchinskii2006}
{\sc V.~Koltchinskii}, {\em Local {Rademacher} complexities and oracle
  inequalities in risk minimization}, The Annals of Statistics, 34 (2006).

\bibitem{Kushner2001}
{\sc H.~J. Kushner and P.~Dupuis}, {\em Numerical Methods for Stochastic
  Control Problems in Continuous Time}, {Springer New York}, 2001.

\bibitem{LeGall2016}
{\sc J.~F. Le~Gall}, {\em Brownian Motion, Martingales, and Stochastic
  Calculus}, {Springer}, 2016.

\bibitem{Leparoux2022}
{\sc C.~Leparoux, B.~H{\'e}riss{\'e}, and F.~Jean}, {\em Structure of optimal
  control for planetary landing with control and state constraints}, {ESAIM:
  Control, Optimisation \& Calculus of Variations}, 28 (2022).

\bibitem{Matheron1975}
{\sc G.~Matheron}, {\em Random sets and integral geometry}, Wiley Series in
  Probability and Mathematical Statistics, 1975.

\bibitem{Mesbah2016}
{\sc A.~Mesbah}, {\em Stochastic model predictive control: An overview and
  perspectives for future research}, {IEEE Control Systems Letters}, 36 (2016),
  pp.~30--44.

\bibitem{Molchanov_BookTheoryOfRandomSets2017}
{\sc I.~Molchanov}, {\em Theory of Random Sets}, {Springer-Verlag}, second~ed.,
  2017.

\bibitem{Pagnoncelli2009}
{\sc B.~K. Pagnoncelli, S.~Ahmed, and A.~Shapiro}, {\em Sample average
  approximation method for chance constrained programming: Theory and
  applications}, Journal of Optimization Theory and Applications, 142 (2009),
  pp.~399--416.

\bibitem{Peng1990}
{\sc S.~Peng}, {\em A general stochastic maximum principle for optimal control
  problems}, {SIAM Journal on Control and Optimization}, 28 (1990),
  pp.~966--979.

\bibitem{Rockafellar2000}
{\sc R.~T. Rockafellar and S.~Uryasev}, {\em Optimization of conditional
  value-at-risk}, {Journal of Risk}, 2 (2000), pp.~21--41.

\bibitem{ShalevShwartz2009}
{\sc S.~Shalev-Shwartz and S.~Ben-David}, {\em Understanding Machine Learning},
  Cambridge University Press, 2009.

\bibitem{Shapiro2014}
{\sc A.~Shapiro, D.~Dentcheva, and A.~Ruszczy{\'{n}}ski}, {\em Lectures on
  Stochastic Programming: Modeling and Theory, Second Edition}, Society for
  Industrial and Applied Mathematics, 2014.

\bibitem{Steltzner2014}
{\sc A.~D. Steltzner, A.~M. San~Martin, T.~P. Rivellini, A.~Chen, and D.~Kipp},
  {\em Mars science laboratory entry, descent, and landing system development
  challenges}, {AIAA Journal of Spacecraft and Rockets}, 51 (2014),
  pp.~994--1003.

\bibitem{Szmuk2016}
{\sc M.~Szmuk, B.~Acikmese, and A.~W. Berning}, {\em Successive convexification
  for fuel-optimal powered landing with aerodynamic drag and non-convex
  constraints}, in {AIAA Conf.\ on Guidance, Navigation and Control}, 2016.

\bibitem{Talagrand2014}
{\sc M.~Talagrand}, {\em Upper and Lower Bounds for Stochastic Processes},
  {Springer Berlin Heidelberg}, 2014.

\bibitem{Vogel2006}
{\sc S.~Vogel}, {\em Semiconvergence in distribution of random closed sets with
  application to random optimization problems}, {Annals of Operations
  Research}, 142 (2006), pp.~269--282.

\bibitem{ipopt2006}
{\sc A.~W\"achter and L.~T. Biegler}, {\em On the implementation of an
  interior-point filter line-search algorithm for large-scale nonlinear
  programming}, {Mathematical Programming}, 106 (2006), pp.~25--57.

\bibitem{wainwright_2019}
{\sc M.~J. Wainwright}, {\em High-Dimensional Statistics: A Non-Asymptotic
  Viewpoint}, Cambridge Series in Statistical and Probabilistic Mathematics,
  Cambridge University Press, 2019.

\bibitem{Wang2008}
{\sc W.~Wang and S.~Ahmed}, {\em Sample average approximation of expected value
  constrained stochastic programs}, {Operations Research Letters}, 36 (2008),
  pp.~515--519.

\bibitem{Yong1999}
{\sc J.~Yong and X.~Y. Zhou}, {\em Stochastic Controls}, {Springer New York},
  1999.

\end{thebibliography}
\end{document}